\documentclass[3p, 12pt]{elsarticle}
\usepackage{amssymb}
\usepackage{amsthm}
\usepackage{appendix}
\usepackage{verbatim}

\usepackage{amsfonts,amsmath,latexsym}
\usepackage[active]{srcltx}
\usepackage{enumitem}

\newtheorem{theorem}{Theorem}[section]
\newtheorem{problem}[theorem]{Problem}
\newtheorem{remark}[theorem]{Remark}

\newtheorem{corollary}[theorem]{Corollary}

\newproof{pf}{Proof}

\newcommand{\RR}{{\if mm {\rm I}\mkern -3mu{\rm R}\else \leavevmode
\hbox{I}\kern -.17em\hbox{R} \fi}}

\newcommand{\bu}{\mbox{\boldmath{$u$}}}
\newcommand{\bt}{\mbox{\boldmath{$t$}}}
\newcommand{\bcero}{\mbox{\boldmath{$0$}}}

\newcommand{\bx}{\mbox{\boldmath{$\bx$}}}
\newcommand{\by}{\mbox{\boldmath{$\by$}}}
\newcommand{\bv}{\mbox{\boldmath{$v$}}}

\newcommand{\var}{\varepsilon}

\newcommand{\ba}{\mbox{\boldmath{$a$}}}

\newcommand{\fb}{{{f}}}
\newcommand{\bg}{\mbox{\boldmath{$g$}}}
\renewcommand{\by}{\mbox{\boldmath{$y$}}}
\renewcommand{\bx}{\mbox{\boldmath{$x$}}}
\newcommand{\be}{\mbox{\boldmath{$e$}}}
\newcommand{\bn}{\mbox{\boldmath{$n$}}}
\newcommand{\bh}{{{h}}}
\newcommand{\bbf}{\mbox{\boldmath{$f$}}}
\newcommand{\bbh}{\mbox{\boldmath{$h$}}}

\newcommand{\btheta}{\mbox{\boldmath{$\theta$}}}

\newcommand{\beeta}{\mbox{\boldmath{$\eta$}}}
\newcommand{\bxi}{\mbox{\boldmath{$\xi$}}}

\newcommand{\bTheta}{\mbox{\boldmath{$\Theta$}}}
\newcommand{\bUcal}{\mbox{\boldmath{$\mathcal{U}$}}}

\newcommand{\Gamae}{\Gamma^{\varepsilon }}

\renewcommand{\d}{\partial}
\newcommand{\eij}{e_{i||j}}

\newcommand{\ekl}{e_{k||l}}
\newcommand{\dekl}{\dot{e}_{k||l}}
\newcommand{\eab}{e_{\alpha||\beta}}
\newcommand{\est}{e_{\sigma||\tau}}
\newcommand{\dest}{\dot{e}_{\sigma||\tau}}
\newcommand{\estres}{e_{\sigma||3}}
\newcommand{\destres}{\dot{e}_{\sigma||3}}
\newcommand{\eatres}{e_{\alpha||3}}
\newcommand{\edtres}{e_{3||3}}
\newcommand{\deab}{\dot{e}_{\alpha||\beta}}
\newcommand{\deatres}{\dot{e}_{\alpha||3}}
\newcommand{\dedtres}{\dot{e}_{3||3}}
\newcommand{\gab}{\gamma_{\alpha\beta}}
\newcommand{\gst}{\gamma_{\sigma\tau}}
\newcommand{\rab}{\rho_{\alpha\beta}}
\newcommand{\rst}{\rho_{\sigma\tau}}
\renewcommand{\a}{a^{\alpha\beta\sigma\tau}}
\renewcommand{\b}{b^{\alpha\beta\sigma\tau}}
\renewcommand{\c}{c^{\alpha\beta\sigma\tau}}
\newcommand{\aeps}{a^{\alpha\beta\sigma\tau,\varepsilon}}
\newcommand{\beps}{b^{\alpha\beta\sigma\tau,\varepsilon}}
\newcommand{\ceps}{c^{\alpha\beta\sigma\tau,\varepsilon}}
\renewcommand{\ae}{\ a.e. \ t\in(0,T)}
\newcommand{\aes}{\ a.e. \ \textrm{in} \ (0,T)}
\newcommand{\forallt}{\ \forall  \ t\in[0,T]}
\newcommand{\Forallt}{\ \forall  \ t\in[0,T]}

\newcommand{\en}{ \ \textrm{in} \ }
\newcommand{\on}{ \ \textrm{on} \ }
\newcommand{\into}{\int_{\omega}}
\newcommand{\intO}{\int_{\Omega}}
\newcommand{\intG}{\int_{\Gamma_+\cup\Gamma_-}}
\newcommand{\ten}{(a^{\alpha \sigma}a^{\beta \tau} + a^{\alpha \tau}a^{\beta\sigma})}
\newcommand{\calQ}{ \mathcal{Q}}

\journal{Journal of gggg}

\begin{document}

\begin{frontmatter}

\title{Derivation of models for linear viscoelastic shells by using asymptotic analysis}

\author[compostela]{G.~Casti\~neira}
\ead{gonzalo.castineira@usc.es}
\author[corunha]{\'A. Rodr\'{\i}guez-Ar\'os}
\ead{angel.aros@udc.es}

\address[compostela]{Departamento de Matem\'atica Aplicada,
Univ. de Santiago de Compostela, Spain}
\address[corunha]{Departamento de M\'etodos Matem\'aticos
e Representaci\'on, Univ. da Coru\~na, Spain}


\begin{abstract}
We consider a family of linear viscoelastic shells with thickness $2\var$ (where $\var$ is a small parameter), clamped along a portion of their lateral face, all having the same middle surface $S$. We formulate the three-dimensional mechanical problem in curvilinear coordinates and provide existence and uniqueness of (weak) solution of the corresponding three-dimensional variational problem. 

We are interested in studying the limit behavior of both the three-dimensional problems and their solutions (displacements $\bu^\var$ of covariant components $u_i^\var$) when $\var$ tends to  zero. To do that, we use asymptotic analysis methods. First, we formulate the variational problem in a fixed domain  independent of $\var$. Then we assume an asymptotic expansion of the scaled displacements field $\bu(\var)=(u_i(\var))$. Identifying the terms of the proposed asymptotic expansion  we characterize the zeroth order term as the solution of a two-dimensional  scaled limit problem. Moreover, on one hand, we find that if the applied body force density is $O(1)$ with respect to $\var$ and surface tractions density is $O(\var)$, the limit of the field $\bu(\var)$ is the solution of a two-dimensional system of variational equations called viscoelastic membrane problem. On the other hand, if the applied body force density is $O(\var^2)$ and surface tractions density is $O(\var^3)$, the limit of the field $\bu(\var)$ is the solution of a different system of two-dimensional variational equations called viscoelastic flexural problem.

In both cases, we find a model which presents a long-term memory that takes into account the deformations at previous times. We finally comment on the existence and uniqueness of solution for the two-dimensional variational problems found and announce convergence results in forthcoming papers.
\end{abstract}

\begin{keyword}
Asymptotic Analysis\sep Viscoelasticity \sep Shells \sep Membrane \sep Flexural \sep Time dependent

34K25, 35O30, 35Q74, 34E05, 34E10, 41A60, 74K25, 74K15, 74D05, 35J15
\end{keyword}

\end{frontmatter}

\section{Introduction}\setcounter{equation}{0}

In solid mechanics, the obtention of  models for rods, beams,
plates and shells is based on {\it a priori} hypotheses on the
displacement and/or stress fields which, upon substitution in the three-dimensional
equilibrium and constitutive equations, lead to useful simplifications. Nevertheless, from both
constitutive and geometrical point of views, there is a need to
justify the validity of most of the models obtained in this way.

For this reason a considerable effort has been made in the past decades by many authors in order
to derive new models and justify the existing ones by
using the asymptotic expansion method, whose foundations can be
found in \cite{Lions}. Indeed, the first applied results were obtained with the justification of the linearized theory of plate bending in \cite{CD,Destuynder}.

The theories of beam bending and rod stretching also benefited from the extensive use of asymptotic methods and so the justification of the Bernoulli-Navier model for the bending-stretching of elastic thin rods was provided in \cite{BV}. In the following years, the nonlinear case was studied
in \cite{CGLRT2} and the analysis and error estimation of higher-order terms in the asymptotic
expansion of the scaled unknowns was given in \cite{iv}. In \cite{TraViano}, the authors use the asymptotic method to justify the Saint-Venant, Timoshenko and Vlassov models of elastic beams.

A description of the  mathematical models for the three-dimensional elasticity, including the nonlinear aspects, together with a mathematical analysis of these models, can be found in \cite{Ciarlet2}. A justification of the two-dimensional equations of a linear plate can be found in  \cite{CD}. An extensive review concerning  plate models can be found in \cite{Ciarlet3}, which also contains the justification of  the models by using asymptotic methods. The existence and uniqueness of solution of elliptic membrane shell equations, can be found in \cite{CiarletLods} and in \cite{CiarletLods2}.  These two-dimensional models are completely justified with convergence theorems. A complete theory regarding  elastic shells can be found in \cite{Ciarlet4b}, where models  for elliptic membranes, generalized membranes and flexural shells are presented. It contains a full description of the asymptotic procedure that leads to the corresponding sets of two-dimensional equations. Also, the dynamic case has been study in \cite{Limin_mem,Limin_flex,Limin_koit}, concerning the justification of dynamic equations for membrane, flexural and Koiter shells. More recently in \cite{ArosObs} the obstacle problem for an elastic elliptic membrane has been identified and justified as the limit problem for a family of unilateral contact problems for elastic elliptic shells.

A large number of real problems had made it necessary the study of new models which could take into account effects such as hardening and memory of the material. An example of these, are the  viscoelasticity models (see \cite{DL,LC1990,Pipkin}). Regarding the obtention and justification of viscoelastic models by using asymptotic expansion methods,  we find  several models for the bending-stretching of
viscoelastic rods in \cite{AV,AV2}. For a family of  shells made of a long-term memory viscoelastic material  we can find in  \cite{Viscoshells,Viscoshellsf,ViscoshellsK} the use of asymptotic analysis to justify with convergence results the limit two-dimensional membrane, flexural and Koiter equations.

In this work, we analyse the asymptotic behaviour of the scaled three-dimensional displacement field of a  shell made of a viscoelastic short-term memory  material (Kelvin-Voigt) as the thickness approaches zero. We consider that the displacements vanish in a portion of the lateral face of the shell, obtaining the equations of a viscoelastic membrane shell or of a viscoelastic flexural shell depending on the order of the forces and the geometry. 
We will follow the notation and style of \cite{Ciarlet4b}, where the linear elastic shells are studied.
For this reason, we shall  reference auxiliary results which apply in the same manner to the viscoelastic case. One of the major differences with respect to previous works in elasticity, consists on time dependence, that will lead to ordinary differential equations that need to be solved in order to find the zeroth-order approach of the solution. 
The structure of the paper is the following: in Section \ref{problema} 
we shall describe the mechanical problem in the original domain, while in Section \ref{seccion_dominio_ind} we will use a projection map into a reference domain, we will introduce the scaled unknowns and forces and the assumptions on coefficients. In Section \ref{preliminares} we recall some technical results which will be needed in what follows and moreover, we include  the theoretical results that support existence and uniqueness of solution for the problems presented in this paper.   In Section \ref{procedure} we show the asymptotic analysis leading to the formulation of the variational equations of the viscoelastic shells. In Section \ref{Existencia} we first recall the classification of the shells attending to its boundary conditions and the geometry of the middle surface $S$ and then, we study the existence and uniqueness of solution of the de-scaled problems derived from the asymptotic procedure.  In Section \ref{conclusiones} we shall present some conclusions, including a comparison between the viscoelastic models and the elastic case studied in \cite{Ciarlet4b} and announce the convergence results in forthcoming papers.

\section{The three-dimensional shell problem}\setcounter{equation}{0} \label{problema}

 We denote by $\mathbb{S}^d$, where $d=2,3$ in practice, the space of second-order symmetric tensors on $\mathbb{R}^d$, while \textquotedblleft$\ \cdot$ \textquotedblright will represent the inner product and $|\cdot|$  the usual norm in $\mathbb{S}^d$ and  $\mathbb{R}^d$. In  what follows, unless the contrary is explicitly written, we will use summation convention on repeated indices. Moreover, Latin indices $i,j,k,l,...$, take their values in the set $\{1,2,3\}$, whereas Greek indices $\alpha,\beta,\sigma,\tau,...$, do it in the set  $\{1,2\}$. Also, we use standard notation for the Lebesgue and Sobolev spaces. Also, for a time dependent function $u$, we denote $\dot{u}$ the first derivative of $u$ with respect to the time variable. 
 
 
 Let ${\Omega}^*$ be a domain of $\mathbb{R}^3$, with a Lipschitz-continuous boundary ${\Gamma^*}=\d{\Omega^*}$. Let ${\bx^*}=({x}_i^*)$ be a generic point of  its closure $\bar{\Omega}^*$ and let ${\d}^*_i$ denote the partial derivative with respect to ${x}_i^*$. Let $d\bx^*$ denote the volume element in $\Omega^*$,  $d\Gamma^*$ denote the area element along $\Gamma^*$ and  $\bn^*$ denote the unit outer normal vector along $\Gamma^*$. Finally, let $\Gamma^*_0$  and $\Gamma_1^*$ be subsets of $\Gamma^*$ such that $meas(\Gamma_0^*)>0$ and $\Gamma^*_0 \cap \Gamma_1^*=\emptyset.$ 
 
  The set $\Omega^*$ is the region occupied by a deformable body in the absence of applied forces. We assume that this body is made of a Kelvin-Voigt viscoelastic material, which is homogeneous and isotropic,  so that the material is characterized by its Lam\'e coefficients   $\lambda\geq0, \mu>0$ and its viscosity coefficients, $\theta\geq 0,\rho\geq 0$ (see for instance \cite{DL,LC1990,Shillor}).

 Let $T>0$ be the time period of observation. Under the effect of applied forces, the  body is deformed and we denote by $u_i^*:[0,T]\times \bar{\Omega}^*\rightarrow \mathbb{R}^3$ the Cartesian components of the displacements field, defined as $\bu^*:=u_i^* \be^{i}:[0,T]\times\bar{\Omega}^* \rightarrow \mathbb{R}^3$, where $\{\be^i\}$ denotes the Euclidean canonical basis in $\mathbb{R}^3$. 
 Moreover, we consider that the displacement field vanishes on the set $\Gamma^*_0$. Hence, the  displacements field $\bu^*=(u_i^*):[0,T]\times\Omega^*\longrightarrow \mathbb{R}^3$ is solution of the following three-dimensional problem in Cartesian coordinates.
 
 \begin{problem}\label{problema_mecanico}
 Find $\bu^*=(u_i^*):[0,T]\times\Omega^*\longrightarrow \mathbb{R}^3$ such that,
 \begin{align}\label{equilibrio}
 -\d_j^*\sigma^{ij,*}(\bu^*)&=f^{i,*} \en \Omega^*, \\\label{Dirichlet}
 u_i^*&=0 \on \Gamma^*_0, \\\label{Neumann}
 \sigma^{ij,*}(\bu^*)n_j^*&=h^{i,*} \on \Gamma_1^*,\\ \label{condicion_inicial} 
 \bu^*(0,\cdot)&=\bu_0^* \en \Omega^*,
 \end{align}
 where the functions
 \begin{align*}
 \sigma^{ij,*}(\bu^*):=A^{ijkl,*}e_{kl}^*(\bu^*)+ B^{ijkl,*}e_{kl}^*(\dot{\bu}^*),
 \end{align*}
 are the  components of the linearized stress tensor field and where the functions
  \begin{align*} 
  & A^{ijkl,*}:= \lambda \delta^{ij}\delta^{kl} + \mu\left(\delta^{ik}\delta^{jl} + \delta^{il}\delta^{jk}\right) , 
  \\ 
  & B^{ijkl,*}:= \theta \delta^{ij}\delta^{kl} + \frac{\rho}{2}\left(\delta^{ik}\delta^{jl} + \delta^{il}\delta^{jk}\right) , 
 \end{align*}
  are the  components of the three-dimensional elasticity and viscosity fourth order tensors, respectively, and 
 \begin{align*}
  e^*_{ij}(\bu^*):= \frac1{2}(\d^*_ju^*_{i}+ \d^*_iu^*_{j}),
 \end{align*}
 designates the  components of the linearized strain tensor associated with the displacement field $\bu^*$of the set $\bar{\Omega}^*$.
 \end{problem}
 We now proceed to describe the equations in Problem \ref{problema_mecanico}. Expression (\ref{equilibrio}) is the equilibrium equation, where $f^{i,*}$ are the  components of the volumic force densities. The equality (\ref{Dirichlet}) is the Dirichlet condition of place, (\ref{Neumann}) is the Neumann condition, where $h^{i,*}$ are the  components of surface force densities and (\ref{condicion_inicial}) is the initial condition, where $\bu_0^*$ denotes the initial  displacements.
 
  Note that, for the sake of briefness, we omit the explicit dependence on the space and time variables when there is no ambiguity. Let us define the space of admissible unknowns,
 \begin{align*} 
 V(\Omega^*)=\{\bv^*=(v_i^*)\in [H^1(\Omega^*)]^3; \bv^*=\mathbf{\bcero} \ on \ \Gamma_0^*  \}.
 \end{align*}
 Therefore, assuming enough regularity,  the unknown  $\bu^*=(u_i^*)$ satisfies the following variational problem in Cartesian coordinates:
\begin{problem}\label{problema_cartesian}
Find $\bu^*=(u_i^*):[0,T]\times {\Omega}^* \rightarrow \mathbb{R}^3$  such that, 
\begin{align*} 
  \displaystyle  \nonumber
  & \bu^*(t,\cdot)\in V(\Omega^*) \forallt,
  \\ \nonumber 
   &\int_{\Omega^*}A^{ijkl,*}e^*_{kl}(\bu^*(t))e^*_{ij}(\bv^*) dx^*+ \int_{\Omega^*} B^{ijkl,*}e^*_{kl}(\dot{\bu}^*(t))e_{ij}^*(\bv^*)   dx^*
  \\ 
 & \quad= \int_{\Omega^*} f^{i,*}(t) v_i^*  dx^* + \int_{\Gamma_1^*} h^{i,*}(t) v_i^*  d\Gamma^* \quad \forall \bv^*\in V(\Omega^*), \aes,
  \\\displaystyle 
  & \bu^*(0,\cdot)= \bu_0^*(\cdot).
\end{align*}
\end{problem} 

Let us consider that $\Omega^*$ is a viscoelastic shell of thickness $2\var$ and middle surface $S$. Now, we shall express the equations of the Problem \ref{problema_cartesian} in terms of  curvilinear coordinates. Let $\omega$ be a domain of $\mathbb{R}^2$, with a Lipschitz-continuous boundary $\gamma=\d\omega$. Let $\by=(y_\alpha)$ be a generic point of  its closure $\bar{\omega}$ and let $\d_\alpha$ denote the partial derivative with respect to $y_\alpha$. 

Let $\btheta\in\mathcal{C}^2(\bar{\omega};\mathbb{R}^3)$ be an injective mapping such that the two vectors $\ba_\alpha(\by):= \d_\alpha \btheta(\by)$ are linearly independent. These vectors form the covariant basis of the tangent plane to the surface $S:=\btheta(\bar{\omega})$ at the point $\btheta(\by)=\by^*.$ We can consider the two vectors $\ba^\alpha(\by)$ of the same tangent plane defined by the relations $\ba^\alpha(\by)\cdot \ba_\beta(\by)=\delta_\beta^\alpha$, that constitute the contravariant basis. We define the unit vector, 
\begin{align}\label{a_3}
\ba_3(\by)=\ba^3(\by):=\frac{\ba_1(\by)\wedge \ba_2(\by)}{| \ba_1(\by)\wedge \ba_2(\by)|},
\end{align} 
 normal vector to $S$ at the point $\btheta(\by)=\by^*$, where $\wedge$ denotes vector product in $\mathbb{R}^3.$ 

We can define the first fundamental form, given as metric tensor, in covariant or contravariant components, respectively, by
\begin{align*}
a_{\alpha\beta}:=\ba_\alpha\cdot \ba_\beta, \qquad a^{\alpha\beta}:=\ba^\alpha\cdot \ba^\beta,
\end{align*}
  the second fundamental form, given as curvature tensor, in covariant or mixed components, respectively, by
\begin{align*}
b_{\alpha\beta}:=\ba^3 \cdot \d_\beta \ba_\alpha, \qquad b_{\alpha}^\beta:=a^{\beta\sigma} b_{\sigma\alpha},
\end{align*}
and the Christoffel symbols of the surface $S$ by
\begin{align*}
\Gamma^\sigma_{\alpha\beta}:=\ba^\sigma\cdot \d_\beta \ba_\alpha.
\end{align*}

The area element along $S$ is $\sqrt{a}dy=dy^*$ where 
\begin{align}\label{definicion_a}
a:=\det (a_{\alpha\beta}).
\end{align}

Let $\gamma_0$ be a subset  of  $\gamma$, such that $meas (\gamma_0)>0$. 
For each $\varepsilon>0$, we define the three-dimensional domain $\Omega^\varepsilon:=\omega \times (-\varepsilon, \varepsilon)$ and  its boundary $\Gamae=\d\Omega^\var$. We also define  the following parts of the boundary, 
\begin{align*}
\Gamma^\varepsilon_+:=\omega\times \{\varepsilon\}, \quad \Gamma^\varepsilon_-:= \omega\times \{-\varepsilon\},\quad \Gamma_0^\varepsilon:=\gamma_0\times[-\varepsilon,\varepsilon].
\end{align*}

Let $\bx^\varepsilon=(x_i^\varepsilon)$ be a generic point of $\bar{\Omega}^\varepsilon$ and let $\d_i^\var$ denote the partial derivative with respect to $x_i^\varepsilon$. Note that $x_\alpha^\varepsilon=y_\alpha$ and $\d_\alpha^\varepsilon =\d_\alpha$. Let $\bTheta:\bar{\Omega}^\varepsilon\rightarrow \mathbb{R}^3$ be the mapping defined by
\begin{align} \label{bTheta}
\bTheta(\bx^\varepsilon):=\btheta(\by) + x_3^\varepsilon \ba_3(\by) \ \forall \bx^\varepsilon=(\by,x_3^\varepsilon)=(y_1,y_2,x_3^\varepsilon)\in\bar{\Omega}^\varepsilon.
\end{align}

The next theorem shows that if the injective mapping $\btheta:\bar{\omega}\rightarrow\mathbb{R}^3$ is smooth enough, the mapping $\bTheta:\bar{\Omega}^\var\rightarrow\mathbb{R}^3$ is also injective for $\var>0$ small enough (see Theorem 3.1-1, \cite{Ciarlet4b}).

\begin{theorem}\label{var_0}
Let $\omega$ be a domain in $\mathbb{R}^2$. Let $\btheta\in\mathcal{C}^2(\bar{\omega};\mathbb{R}^3)$ be an injective mapping such that the two vectors $\ba_\alpha=\d_\alpha\btheta$ are linearly independent at all points of $\bar{\omega}$ and let $\ba_3$,  defined  in (\ref{a_3}). Then there exists $\var_0>0$ such that   the mapping $\bTheta:\bar{\Omega}_0 \rightarrow\mathbb{R}^3$ defined by
\begin{align*}
\bTheta(\by,x_3):=\btheta(\by) + x_3 \ba_3(\by) \ \forall (\by,x_3)\in\bar{\Omega}_0, \ \textrm{where} \ \Omega_0:=\omega\times(-\var_0,\var_0),
\end{align*}
is a $\mathcal{C}^1-$ diffeomorphism from $\bar{\Omega}_0$ onto $\bTheta(\bar{\Omega}_0)$ and $\det (\bg_1,\bg_2,\bg_3)>0$ in $\bar{\Omega}_0$, where $\bg_i:=\d_i\bTheta$. 
\end{theorem}

For each $\var$, $0<\var\le\var_0$, the set $\bTheta(\bar{\Omega}^\var)=\bar{\Omega}^*$ is the reference configuration of a viscoelastic shell, with middle surface $S=\btheta(\bar{\omega})$ and thickness $2\varepsilon>0$.
Furthermore for $\varepsilon>0,$ $\bg_i^\varepsilon(\bx^\varepsilon):=\d_i^\varepsilon\bTheta(\bx^\varepsilon)$ are linearly independent and the mapping $\bTheta:\bar{\Omega}^\varepsilon\rightarrow \mathbb{R}^3$ is injective for all $\var$, $0<\var\le\var_0$, as a consequence of injectivity of the mapping $\btheta$. Hence, the three vectors $\bg_i^\varepsilon(\bx^\varepsilon)$ form the covariant basis of the tangent space at the point $\bx^*=\bTheta(\bx^\varepsilon)$ and $\bg^{i,\varepsilon}(\bx^\varepsilon) $ defined by the relations $\bg^{i,\varepsilon}\cdot \bg_j^\varepsilon=\delta_j^i$ form the contravariant basis at the point $\bx^*=\bTheta(\bx^\varepsilon)$. We define the metric tensor, in covariant or contravariant components, respectively, by
\begin{align*}
 g_{ij}^\varepsilon:=\bg_i^\varepsilon \cdot \bg_j^\varepsilon,\quad g^{ij,\varepsilon}:=\bg^{i,\varepsilon} \cdot \bg^{j,\varepsilon},
\end{align*}
and Christoffel symbols by
\begin{align} \label{simbolos3D}
\Gamma^{p,\varepsilon}_{ij}:=\bg^{p,\varepsilon}\cdot\d_i^\varepsilon \bg_j^\varepsilon. 
\end{align}

The volume element in the set $\bTheta(\bar{\Omega}^\varepsilon)=\bar{\Omega}^*$ is $\sqrt{g^\varepsilon}dx^\var=dx^*$ and the surface element in $\bTheta(\Gamma^\varepsilon)=\Gamma^*$ is $\sqrt{g^\varepsilon}d\Gamae =d\Gamma^*$  where
\begin{align} \label{g}
g^\varepsilon:=\det (g^\varepsilon_{ij}).
\end{align} 
Therefore, for a field ${\bv}^*$ defined in $\bTheta(\bar{\Omega}^\var)=\bar{\Omega}^*$, we define its covariant curvilinear coordinates $v_i^\var$ by
\begin{equation*}
{\bv}^*({\bx}^*)={v}^*_i({\bx}^*){\be}^i=:v_i^\var(\bx^\var)\bg^i(\bx^\var),\ {\rm with}\ {\bx}^*=\bTheta(\bx^\var).
\end{equation*}


Besides, we denote by $u_i^\varepsilon:[0,T]\times \bar{\Omega}^\varepsilon \rightarrow \mathbb{R}^3$ the covariant components of the displacements field, that is  $\bUcal^\var:=u_i^\varepsilon \bg^{i,\varepsilon}:[0,T]\times\bar{\Omega}^\varepsilon \rightarrow \mathbb{R}^3$ . For simplicity, we define the vector field $\bu^\varepsilon=(u_i^\varepsilon):[0,T]\times {\Omega}^\varepsilon \rightarrow \mathbb{R}^3$ which will be denoted vector of unknowns.

Recall that we assumed that the shell is subjected to a boundary condition of place; in particular that the displacements field vanishes in a portion of the lateral face of the shell, that is, $\bTheta(\Gamma_0^\varepsilon)=\Gamma_0^*$.

Accordingly, let us define the space of admissible unknowns,
\begin{align*} 
V(\Omega^\varepsilon)=\{\bv^\varepsilon=(v_i^\varepsilon)\in [H^1(\Omega^\varepsilon)]^3; \bv^\varepsilon=\mathbf{\bcero} \ on \ \Gamma_0^\varepsilon  \}.
\end{align*}

This is a real Hilbert space with the induced inner product of $[H^1(\Omega^\var)]^3$. The corresponding norm  is denoted by $||\cdot||_{1,\Omega^\var}$. 

Therefore, we can find the expression of the Problem \ref{problema_cartesian} in curvilinear coordinates (see \cite{Ciarlet4b} for details). Hence, the `` displacements " field $\bu^\var=(u_i^\var)$ verifies the following variational problem of a three-dimensional viscoelastic shell in curvilinear coordinates:

\begin{problem}\label{problema_eps}
Find $\bu^\varepsilon=(u_i^\varepsilon):[0,T]\times {\Omega}^\varepsilon \rightarrow \mathbb{R}^3$  such that, 
\begin{align} 
  \displaystyle  \nonumber
  & \bu^\varepsilon(t,\cdot)\in V(\Omega^\varepsilon) \forallt,
  \\ \nonumber 
   &\int_{\Omega^\varepsilon}A^{ijkl,\varepsilon}e^\varepsilon_{k||l}(\bu^\varepsilon(t))e^\varepsilon_{i||j}(\bv^\varepsilon)\sqrt{g^\varepsilon} dx^\varepsilon+ \int_{\Omega^\varepsilon} B^{ijkl,\varepsilon}e^\varepsilon_{k||l}(\dot{\bu}^\varepsilon(t))e_{i||j}^\var(\bv^\varepsilon) \sqrt{g^\varepsilon}  dx^\varepsilon
  \\ \label{Pbvariacionaleps}
 & \quad= \int_{\Omega^\varepsilon} f^{i,\varepsilon}(t) v_i^\varepsilon \sqrt{g^\varepsilon} dx^\varepsilon + \int_{\Gamma_+^\varepsilon\cup\Gamma_-^\varepsilon} h^{i,\varepsilon}(t) v_i^\varepsilon\sqrt{g^\varepsilon}  d\Gamma^\varepsilon  \quad \forall \bv^\varepsilon\in V(\Omega^\varepsilon), \aes,
  \\\displaystyle \nonumber
  & \bu^\varepsilon(0,\cdot)= \bu_0^\varepsilon(\cdot),
\end{align}
\end{problem}
 where the functions
  \begin{align}\label{TensorAeps}
  & A^{ijkl,\varepsilon}:= \lambda g^{ij,\varepsilon}g^{kl,\varepsilon} + \mu(g^{ik,\varepsilon}g^{jl,\varepsilon} + g^{il,\varepsilon}g^{jk,\varepsilon} ), 
  \\ \label{TensorBeps}
  & B^{ijkl,\varepsilon}:= \theta g^{ij,\varepsilon}g^{kl,\varepsilon} + \frac{\rho}{2}(g^{ik,\varepsilon}g^{jl,\varepsilon} + g^{il,\varepsilon}g^{jk,\varepsilon} ), 
 \end{align}
  are the contravariant components of the three-dimensional elasticity and viscosity tensors, respectively. We assume that the Lam\'e coefficients   $\lambda\geq0, \mu>0$ and the viscosity coefficients $\theta\geq 0,\rho\geq 0$  are all independent of $\var$. Moreover, the terms
 \begin{align*}
  e^\varepsilon_{i||j}(\bu^\var):= \frac1{2}(u^\varepsilon_{i||j}+ u^\varepsilon_{j||i})=\frac1{2}(\d^\varepsilon_ju^\varepsilon_i + \d^\varepsilon_iu^\varepsilon_j) - \Gamma^{p,\varepsilon}_{ij}u^\varepsilon_p,
 \end{align*}
 designate the covariant components of the linearized strain tensor associated with the displacement field $\bUcal^\var$of the set $\bTheta(\bar{\Omega}^\varepsilon)$.  Moreover, $f^{i,\var}$ denotes the contravariant components of the volumic force densities, $h^{i,\var}$ denotes contravariant components of surface force densities and $\bu_0^\var$ denotes the initial `` displacements " (actually, the initial displacement is $\bUcal_0^\var:=(u_0^\var)_i\bg^{i,\var}$).

 Note that the following additional relations are satisfied,
 \begin{align}\nonumber
 \Gamma^{3,\varepsilon}_{\alpha 3}=\Gamma^{p,\varepsilon}_{33}&=0  \ \textrm{in} \ \bar{\Omega}^\varepsilon, \\
\label{tensor_terminos_nulos}
 A^{\alpha\beta\sigma 3,\varepsilon}=A^{\alpha 333,\varepsilon}=B^{\alpha\beta\sigma 3 , \varepsilon}&=B^{\alpha 333, \varepsilon}=0 \ \textrm{in} \ \bar{\Omega}^\varepsilon,
 \end{align}
 as a consequence of the definition of $\bTheta$ in (\ref{bTheta}).  The definitions of the fourth order tensors (\ref{TensorAeps}) and (\ref{TensorBeps}), imply that (see Theorem 1.8-1, \cite{Ciarlet4b}) for $\var>0$ small enough, there exist two constants $C_e>0$ and $C_v>0$, independent of $\var$, such that,
  \begin{align} \label{elipticidadA}
  \sum_{i,j}|t_{ij}|^2\leq C_e A^{ijkl,\var}(\bx^\var)t_{kl}t_{ij},\\\label{elipticidadB}
  \sum_{i,j}|t_{ij}|^2\leq C_v B^{ijkl,\var}(\bx^\var)t_{kl}t_{ij},
  \end{align}
 for all $\bx^\var\in\bar{\Omega}^\var$ and all $\bt=(t_{ij})\in\mathbb{S}^2$.
 
 \begin{remark}
 Note that the proof for the scaled viscosity tensor $\left(B^{ijkl,\var}\right)$ would follow the steps of the proof for the elasticity tensor $\left(A^{ijkl,\var} \right)$ in Theorem 1.8-1, \cite{Ciarlet4b}, since from a quality point of view their expressions differ in replacing the Lam\'e constants by the two viscosity coefficients. 
 \end{remark}

The proof that Problem \ref{problema_eps} has a unique  solution for $\var>0$ small enough is left to Section \ref{preliminares} (see Theorem \ref{Thexistunic}).

\section{The scaled three-dimensional shell problem}\setcounter{equation}{0} \label{seccion_dominio_ind}

For convenience, we consider a reference domain independent of the small parameter $\var$. Hence, let us define the three-dimensional domain $\Omega:=\omega \times (-1, 1) $ and  its boundary $\Gamma=\d\Omega$. We also define the following parts of the boundary,
 \begin{align*}
 \Gamma_+:=\omega\times \{1\}, \quad \Gamma_-:= \omega\times \{-1\},\quad \Gamma_0:=\gamma_0\times[-1,1].
 \end{align*}
 Let $\bx=(x_1,x_2,x_3)$ be a generic point in $\bar{\Omega}$ and we consider the notation $\d_i$ for the partial derivative with respect to $x_i$. We define the following projection map, 
 \begin{align*}
 \pi^\varepsilon:\bx=(x_1,x_2,x_3)\in \bar{\Omega} \longrightarrow \pi^\varepsilon(\bx)=\bx^\varepsilon=(x_i^\varepsilon)=(x_1^\var,x_2^\var,x_3^\var)=(x_1,x_2,\varepsilon x_3)\in \bar{\Omega}^\varepsilon,
 \end{align*}
 hence, $\d_\alpha^\varepsilon=\d_\alpha $  and $\d_3^\varepsilon=\frac1{\varepsilon}\d_3$. We consider the scaled unknown $\bu(\varepsilon)=(u_i(\varepsilon)):[0,T]\times \bar{\Omega}\longrightarrow \mathbb{R}^3$ and the scaled vector fields $\bv=(v_i):\bar{\Omega}\longrightarrow \mathbb{R}^3 $ defined as
 \begin{align*}
 u_i^\varepsilon(t,\bx^\varepsilon)=:u_i(\varepsilon)(t,\bx) \ \textrm{and} \ v_i^\varepsilon(\bx^\varepsilon)=:v_i(\bx) \ \forall \bx^\varepsilon=\pi^\varepsilon(\bx)\in \bar{\Omega}^\varepsilon, \ \forall \ t\in[0,T].
 \end{align*}

We remind that, by hypothesis, the Lam\'e and viscosity constants are independent of $\varepsilon$. Also, let the functions, $\Gamma_{ij}^{p,\varepsilon}, g^\varepsilon, A^{ijkl,\varepsilon}, B^{ijkl,\varepsilon}$ defined in (\ref{simbolos3D}), (\ref{g}), (\ref{TensorAeps}) and (\ref{TensorBeps}), be associated with the functions $\Gamma_{ij}^p(\varepsilon), g(\varepsilon), A^{ijkl}(\varepsilon), B^{ijkl}(\varepsilon)$ defined by
  \begin{align} \label{escalado_simbolos}
  &\Gamma_{ij}^p(\varepsilon)(\bx):=\Gamma_{ij}^{p,\varepsilon}(\bx^\varepsilon),\\\label{escalado_g}
  & g(\varepsilon)(\bx):=g^\varepsilon(\bx^\varepsilon),\\\label{tensorA_escalado}
  & A^{ijkl}(\varepsilon)(\bx):=A^{ijkl,\varepsilon}(\bx^\varepsilon),\\\label{tensorB_escalado}
  & B^{ijkl}(\varepsilon)(\bx):=B^{ijkl,\varepsilon}(\bx^\varepsilon),
  \end{align}
 for all $\bx^\varepsilon=\pi^\varepsilon(\bx)\in\bar{\Omega}^\varepsilon$. For all $\bv=(v_i)\in [H^1(\Omega)]^3$, let there be associated the scaled linearized strains $(\eij(\var)(\bv))\in L^2(\Omega)$, defined by
\begin{align*} 
&\eab(\varepsilon;\bv):=\frac{1}{2}(\d_\beta v_\alpha + \d_\alpha v_\beta) - \Gamma_{\alpha\beta}^p(\varepsilon)v_p,\\ 
  & \eatres(\varepsilon;\bv):=\frac{1}{2}(\frac{1}{\var}\d_3 v_\alpha + \d_\alpha v_3) - \Gamma_{\alpha 3}^p(\varepsilon)v_p,\\ 
  & \edtres(\varepsilon;\bv):=\frac1{\varepsilon}\d_3v_3.
\end{align*}
Note that with these definitions it is verified that
\begin{align*}
\eij^\var(\bv^\var)(\pi^\var(\bx))=\eij(\var;\bv)(\bx) \ \forall\bx\in\Omega.
\end{align*}

 \begin{remark} The functions $\Gamma_{ij}^p(\varepsilon), g(\varepsilon), A^{ijkl}(\varepsilon), B^{ijkl}(\varepsilon)$ converge in $\mathcal{C}^0(\bar{\Omega})$ when $\varepsilon$ tends to zero.
 \end{remark}
 
 \begin{remark}When we consider
 $\varepsilon=0$ the functions will be defined with respect to $\by\in\bar{\omega}$. We shall distinguish the three-dimensional Christoffel symbols from the two-dimensional ones by using  $\Gamma_{\alpha \beta}^\sigma(\varepsilon)$ and $ \Gamma_{\alpha\beta}^\sigma$, respectively.
 \end{remark}

The next result is an adaptation of $(b)$ in Theorem 3.3-2, \cite{Ciarlet4b} to the viscoelastic case.  We will study the asymptotic behavior of the scaled contravariant components $A^{ijkl}(\var), B^{ijkl}(\var)$ of the three-dimensional elasticity and viscosity tensors defined in (\ref{tensorA_escalado})--(\ref{tensorB_escalado}), as $\var\rightarrow0$.  We show their uniform positive definiteness  not only with respect to $\bx\in\bar{\Omega}$, but also with respect to $\var$, $0<\var\leq\var_0$. Finally, their limits are functions of $\by\in\bar{\omega}$ only, that is, independent of the transversal variable $x_3$.

\begin{theorem} \label{Th_comportamiento asintotico}
Let $\omega$  be a domain in $\mathbb{R}^2$ and let $\btheta\in\mathcal{C}^2(\bar{\omega};\mathbb{R}^3)$ be an injective mapping such that the two vectors $\ba_\alpha=\d_\alpha\btheta$ are linearly independent at all points of $\bar{\omega}$, let $a^{\alpha\beta}$ denote the contravariant components of the metric tensor of $S=\btheta(\bar{\omega})$. In addition to that, let the other assumptions on the mapping $\btheta$ and the definition of $\var_0$ be as in Theorem \ref{var_0}. The contravariant components $A^{ijkl}(\var), B^{ijkl}(\var)$ of the scaled three-dimensional elasticity and viscosity tensors, respectively, defined in (\ref{tensorA_escalado})--(\ref{tensorB_escalado}) satisfy
\begin{align*}
A^{ijkl}(\var)= A^{ijkl}(0) + O(\var) \ \textrm{and} \ A^{\alpha\beta\sigma 3}(\var)=A^{\alpha 3 3 3}(\var)=0, \\
B^{ijkl}(\var)= B^{ijkl}(0) + O(\var) \ \textrm{and} \ B^{\alpha\beta\sigma 3}(\var)=B^{\alpha 3 3 3}(\var)=0 ,
\end{align*}
for all $\var$, $0<\var \leq \var_0$, 
 and
\begin{align*}
A^{\alpha\beta\sigma\tau}(0)&= \lambda a^{\alpha\beta}a^{\sigma\tau} + \mu(a^{\alpha\sigma}a^{\beta\tau} + a^{\alpha\tau}a^{\beta\sigma}), & A^{\alpha\beta 3 3}(0)&= \lambda a^{\alpha\beta},
\\
 A^{\alpha 3\sigma 3}(0)&=\mu a^{\alpha\sigma} ,& A^{33 3 3}(0)&= \lambda + 2\mu,
 \\
A^{\alpha\beta\sigma 3}(0) &=A^{\alpha 333}(0)=0,
\\
B^{\alpha\beta\sigma\tau}(0)&= \theta a^{\alpha\beta}a^{\sigma\tau} + \frac{\rho}{2}(a^{\alpha\sigma}a^{\beta\tau} + a^{\alpha\tau}a^{\beta\sigma}),& B^{\alpha\beta 3 3}(0)&= \theta a^{\alpha\beta},
\\
 B^{\alpha 3\sigma 3}(0)&=\frac{\rho}{2} a^{\alpha\sigma} ,& B^{33 3 3}(0)&= \theta + \rho, 
\\
B^{\alpha\beta\sigma 3}(0) &=B^{\alpha 333}(0)=0.
\end{align*}

 Moreover, there exist two constants $C_e>0$ and $C_v>0$, independent of the variables and $\var$, such that 
  \begin{align} \label{elipticidadA_eps}
  \sum_{i,j}|t_{ij}|^2\leq C_e A^{ijkl}(\varepsilon)(\bx)t_{kl}t_{ij},\\\label{elipticidadB_eps}
  \sum_{i,j}|t_{ij}|^2 \leq C_v B^{ijkl}(\varepsilon)(\bx)t_{kl}t_{ij},
  \end{align}
 for all $\var$, $0<\var\leq\var_0$, for all $\bx\in\bar{\Omega}$ and all $\bt=(t_{ij})\in\mathbb{S}^2$.
\end{theorem}
 
 \begin{remark}
 Note that the proof for the scaled viscosity tensor $\left(B^{ijkl}(\varepsilon)\right)$ would follow the steps of the proof for the elasticity tensor $\left(A^{ijkl}(\var)\right)$ in Theorem 3.3-2, \cite{Ciarlet4b}, since from a quality point of view their expressions differ in replacing the Lam\'e constants by the two viscosity coefficients. 
 \end{remark}
 
\begin{remark}
The asymptotic behavior of $g(\var)$ and the contravariant components of elasticity and viscosity tensors, $A^{ijkl}(\var)$, $B^{ijkl}(\var)$ also implies that
\begin{align} \label{tensorA_tildes}
A^{ijkl}(\var)\sqrt{g(\var)}= A^{ijkl}(0)\sqrt{a} + \var \tilde{A}^{ijkl,1} + \var^2 \tilde{A}^{ijkl,2} + o(\var^2), \\ \label{tensorB_tildes}
B^{ijkl}(\var)\sqrt{g(\var)}= B^{ijkl}(0)\sqrt{a} + \var \tilde{B}^{ijkl,1} + \var^2 \tilde{B}^{ijkl,2} + o(\var^2),
\end{align}
for certain regular contravariant components $\tilde{A}^{ijkl,\alpha}, \tilde{B}^{ijkl,\alpha}$ of certain tensors.
\end{remark}

 Let the scaled applied forces $\bbf(\varepsilon):[0,T]\times \Omega\longrightarrow \mathbb{R}^3$ and  $\bbh(\varepsilon):[0,T]\times (\Gamma_+\cup\Gamma_-)\longrightarrow \mathbb{R}^3$ be defined by
   \begin{align*}
  \bbf^\var&=(f^{i,\varepsilon})(t,\bx^\varepsilon)=:\bbf(\var)= (f^i(\varepsilon))(t,\bx) 
  \\ \nonumber
  &\forall \bx\in\Omega, \ \textrm{where} \ \bx^\varepsilon=\pi^\varepsilon(\bx)\in \Omega^\varepsilon \ \textrm{and} \ \forall t\in[0,T], \\ 
   \bbh^\var&=(h^{i,\varepsilon})(t,\bx^\varepsilon)=:\bbh(\var)= (h^i(\varepsilon))(t,\bx) 
   \\ \nonumber 
   &\forall \bx\in\Gamma_+\cup\Gamma_-, \ \textrm{where} \ \bx^\varepsilon=\pi^\varepsilon(\bx)\in \Gamma_+^\varepsilon\cup\Gamma_-^\varepsilon \ \textrm{and} \ \forall t\in[0,T].
   \end{align*}
   Also, we introduce $\bu_0(\var): \Omega \longrightarrow \mathbb{R}^3$ as
   \begin{align*}
   \bu_0(\var)(\bx):=\bu_0^\var(\bx^\var) \ \forall \bx\in\Omega, \ \textrm{where} \ \bx^\varepsilon=\pi^\varepsilon(\bx)\in \Omega^\varepsilon,
   \end{align*}
   and define the space
   \begin{align*} 
   V(\Omega)=\{\bv=(v_i)\in [H^1(\Omega)]^3; \bv=\mathbf{0} \ on \ \Gamma_0\},
   \end{align*}
  which is a Hilbert space, with associated norm denoted by $||\cdot||_{1,\Omega}$.

   The scaled variational problem  can then  be written as follows:
  \begin{problem}\label{problema_escalado}
  Find $\bu(\varepsilon):[0,T]\times \Omega\longrightarrow \mathbb{R}^3$ such that,  
  \begin{align}\nonumber 
  &\bu(\varepsilon)(t,\cdot)\in V(\Omega) \forallt, \\ \nonumber
      &\int_{\Omega}A^{ijkl}(\varepsilon)e_{k||l}(\varepsilon;\bu(\varepsilon))e_{i||j}(\varepsilon;\bv)\sqrt{g(\varepsilon)} dx
   + \int_{\Omega} B^{ijkl}(\varepsilon)e_{k||l}(\varepsilon;\dot{\bu}(\varepsilon))e_{i||j}(\varepsilon;\bv) \sqrt{g(\varepsilon)}  dx
     \\ \label{ec_problema_escalado}
     & \quad= \int_{\Omega} \fb^{i}(\varepsilon) v_i \sqrt{g(\varepsilon)} dx + \frac{1}{\varepsilon}\int_{\Gamma_+\cup\Gamma_-} \bh^{i}(\varepsilon) v_i\sqrt{g(\varepsilon)}  d\Gamma  \quad \forall \bv\in V(\Omega), \aes,
     \\\displaystyle \nonumber
     & \bu(\var)(0,\cdot)= \bu_0(\var)(\cdot).
    \end{align} 
  \end{problem}

\begin{remark}
Note that the order of the applied forces has not been determined yet.
\end{remark} 

The proof that Problem \ref{problema_escalado} has a unique  solution  is left to Section \ref{preliminares} (see Theorem \ref{Theorema_exist_escalado_sin_orden}).

\section{Technical preliminaries}\setcounter{equation}{0} \label{preliminares}

Concerning geometrical and mechanical preliminaries, we shall present some theorems, which will be used in the following sections. Then, we show some new results related with the existence and uniqueness of solution of the problems presented in this paper. First, we recall the Theorem 3.3-1, \cite{Ciarlet4b}.

\begin{theorem} \label{Th_simbolos2D_3D}
Let $\omega$ be a domain in $\mathbb{R}^2$, let $\btheta\in\mathcal{C}^3(\bar{\omega};\mathcal{R}^3)$ be an injective mapping such that the two vectors $\ba_\alpha=\d_\alpha\btheta$ are linearly independent at all points of $\bar{\omega}$ and let $\var_0>0$ be as in Theorem \ref{var_0}. The functions $\Gamma^p_{ij}(\var)=\Gamma^p_{ji}(\var)$ and $g(\var)$ are defined in (\ref{escalado_simbolos})--(\ref{escalado_g}), the functions $b_{\alpha\beta}, b_\alpha^\sigma, \Gamma_{\alpha\beta}^\sigma,a$, are defined in Section \ref{problema} and the covariant derivatives $b_\beta^\sigma|_\alpha$ are defined by
\begin{align} \label{b_barra}
b_\beta^\sigma|_\alpha:=\d_\alpha b_\beta^\sigma +\Gamma^\sigma_{\alpha\tau}b_\beta^\tau - \Gamma^\tau_{\alpha\beta}b^\sigma_\tau.
\end{align}
The functions $b_{\alpha\beta}, b_\alpha^\sigma, \Gamma_{\alpha\beta}^\sigma, b_\beta^\sigma|_\alpha$ and $a$ are identified with functions in $\mathcal{C}^0(\bar{\Omega})$. Then
\begin{align*}
\begin{aligned}[c]
 \Gamma_{\alpha\beta}^\sigma(\var)&=  \Gamma_{\alpha\beta}^\sigma -\var x_3b_\beta^\sigma|_\alpha + O(\var^2), \\
  \d_3 \Gamma_{\alpha\beta}^p(\var)&= O(\var), 
   \\
   \Gamma_{\alpha3}^3(\var)&=\Gamma_{33}^p(\var)=0,
\end{aligned}
\qquad
\begin{aligned}[c]
 \Gamma_{\alpha\beta}^3(\var)&=b_{\alpha\beta} - \var x_3 b_\alpha^\sigma b_{\sigma\beta}, 
 \\
 \Gamma_{\alpha3}^\sigma(\var)& = -b_\alpha^\sigma - \var x_3 b_\alpha^\tau b_\tau^\sigma + O(\var^2), 
\\
 g(\varepsilon)&=a + O(\varepsilon),
\end{aligned}
\end{align*}
for all $\var$, $0<\var\leq\var_0$, where the order symbols $O(\var)$ and $O(\var^2)$  are meant with respect to the norm $||\cdot||_{0,\infty,\bar{\Omega}}$ defined by
\begin{align*} 
||w||_{0,\infty,\bar{\Omega}}=\sup \{|w(\bx)|; \bx\in\bar{\Omega}\}.
\end{align*}
 Finally, there exist constants $a_0, g_0$ and $g_1$ such that
 \begin{align*}
 & 0<a_0\leq a(\by) \ \forall \by\in \bar{\omega},
 \\ 
 & 0<g_0\leq g(\varepsilon)(\bx) \leq g_1 \ \forall \bx\in\bar{\Omega} \ \textrm{and} \ \forall \ \var, 0<\varepsilon\leq \varepsilon_0.
 \end{align*}
\end{theorem}

We now include the following result that will be used repeatedly in what follows (see Theorem 3.4-1, \cite{Ciarlet4b}, for details).

\begin{theorem} \label{th_int_nula}
 Let $\omega$ be a domain in $\mathbb{R}^2$ with boundary $\gamma$, let $\Omega=\omega\times (-1,1)$, and let $g\in L^p(\Omega)$, $p>1$, be a function such that 
 \begin{align*}
 \intO g \d_3v dx=0, \ \textrm{for all} \ v\in \mathcal{C}^{\infty}(\bar{\Omega}) \ \textrm{with} \ v=0 \on \gamma\times[-1,1]. 
 \end{align*}
 Then $g=0.$
\end{theorem}
\begin{remark}
This result holds if $\intO g \d_3v dx=0$ for all $v\in H^1(\Omega)$ such that $v=0$ in $\Gamma_0$. It is in this way that we will use this result in the following.
\end{remark}

In what follows we shall present several results related with the existence and uniqueness of the solutions of the problems presented in this paper.   Moreover, we show the regularity of these solutions depending on the regularity of the data provided.

Let $V$ be a Hilbert space. 
 We denote by $(\cdot,\cdot)_V$ and $||\cdot||_V$ the corresponding inner product and associated norm. Consider the bounded operators $B:V\longrightarrow V$, $A:V\longrightarrow V$ and a function $f:(0,T)\longrightarrow V$. Let also $u_0\in V$. We are interested in studying the problem

\begin{problem} \label{problema_apendice}
 Find $u: [0,T]\to V$ such that,
\begin{align*}
& B \dot{u}(t)+ A u(t) =f(t) \ae, \\
& u(0)=u_0.   
\end{align*}
\end{problem}

\begin{theorem}\label{teorema_existenciayunicidad}
 Assume that  $B:V\longrightarrow V$ is strongly monotone, Lipschitz-continuous operator and $A:V\longrightarrow V$ is a Lipschitz-continuous operator. Also, let $u_0\in V$ and $f\in L^2(0,T;V)$. Then, the Problem \ref{problema_apendice} has a unique solution
 $u\in W^{1,2}(0,T;V)$.
\end{theorem}

 The proof of this theorem can be found in Theorem 3.3, \cite{pto_fijo}, where the author uses the inverse of the operator $A$ and the Banach fixed point theorem. Alternatively, we can prove  the result without explicitly using the inverse of the operator by using its Lipschitz-continuity instead.
 
 The existence and uniqueness of the inhomogeneous evolutionary equations, when the operator $B$ is the identity, can be found in Chapter 6, \cite{Yosida}. In addition, in \cite{Mascarenhas} the author proves the scalar version  for the quasi-static case and with no body loadings. In  Chapter 6, \cite{SanchezPhy}, it is shown that  these restrictions can be dropped obtaining the existence of a unique solution in the framework of semigroup theory.

 \begin{corollary} \label{Cor_ex_un_reg}
Under the assumptions of the previous theorem if, in addition, $\dot{f}\in L^2(0,T;V)$ and  the operators $A$ and $B$ are linear, the Problem  \ref{problema_apendice} has a unique solution  $u\in{W}^{2,2}(0,T;V)$.
 \end{corollary}
\begin{proof} The existence and uniqueness of $\bu\in W^{1,2}(0,T;V)$ is consequence of the Theorem \ref{teorema_existenciayunicidad}. Let us find the additional regularity of the solution. To do that consider the equation
\begin{equation}\label{Bz}
B\dot{z}(t) + A{z}(t)=\dot{f}(t), \ae,
\end{equation} 
with the initial condition $B{z}(0)=f(0) - A u_0\in V$. By Theorem \ref{teorema_existenciayunicidad} there exists a unique $z\in W^{1,2}(0,T;V)$ solution of (\ref{Bz}).
Now, if we integrate the equation and substitute the initial condition, by the linearity of the operator $B$ we find that
\begin{align*}
B({{z}}(t)) - B({z}(0))+ \int_{0}^{t}A {z}(s)ds= f(t) - f(0).
\end{align*}
Let ${w}(t)=u_0 + \int_{0}^{t}{z}(s)ds$, so that $\dot{w}(t)={z}(t)$ and $w(0)=u_0$. Due to the linearity of the operator $A$ we find that
\begin{align*}
B\dot{w}(t)+ A ({w}(t)-{u}_0)=f(t)-Au_0,
\end{align*}
hence,
\begin{align*}
B\dot{w}(t)+ A {w}(t)=f(t).
\end{align*}
Since by Theorem \ref{teorema_existenciayunicidad}  there is a unique solution for this equation, we deduce that $u={w}\in {W}^{1,2}(0,T;V)$. Moreover, as $z$ is solution of (\ref{Bz}) then $\dot{u}=\dot{w}=z\in{W}^{1,2}(0,T;V)$. Therefore, we conclude $u\in{W}^{2,2}(0,T;V)$.
\end{proof}

\begin{theorem}\label{Thexistunic}
 Let $\Omega^\var$ be a domain in $\mathbb{R}^3$ defined as in Section \ref{problema} and let $\bTheta$ be a  $\mathcal{C}^2$-diffeomorphism of $\bar{\Omega}^\var$ in its image $\bTheta(\bar{\Omega}^\var)$, such that the three vectors $\bg_i^\var(\bx)=\d_i^\var\bTheta(\bx^\var)$ are linearly independent for all $\bx^\var\in\bar{\Omega}^\var$. Let $\Gamma_0^\var$ be a $d\Gamma^\var$-measurable subset of $\Gamma^\var=\d\Omega^\var$ such that  $meas(\Gamma_0^\var)>0.$
 Let $\fb^{i,\var}\in L^{2}(0,T; L^2(\Omega^\var)) $, $\bh^{i,\var}\in L^{2}(0,T; L^2(\Gamma_1^\var))$, where $\Gamma_1^\var:= \Gamma_+^\var\cup\Gamma_-^\var$. Let  $\bu_0^\var\in V(\Omega^\var). $ Then, there exists a unique solution $\bu^\var=(u_i^\var):[0,T]\times\Omega^\var \rightarrow \mathbb{R}^3$ satisfying the Problem \ref{problema_eps}. Moreover $\bu^\var\in W^{1,2}(0,T;V(\Omega^\var))$. In addition to that, if $\dot{\fb}^{i,\var}\in L^{2}(0,T; L^2(\Omega^\var)) $, $\dot{\bh}^{i,\var}\in L^{2}(0,T; L^2(\Gamma_1^\var))$, then $\bu^\var\in W^{2,2}(0,T;V(\Omega^\var))$.
\end{theorem}
\begin{proof}
Let $V=V(\Omega^\var)$ for simplicity.  By the Riesz Representation Theorem we find that there exist bounded linear operators  $B:V\longrightarrow V ,$ $A:V\longrightarrow V$ and $\bbf\in V$ such that
\begin{align*}
(B\bu^\var,\bv^\var)_{V}&:=\int_{\Omega^\varepsilon} B^{ijkl,\varepsilon}e^\varepsilon_{k||l}(\dot{\bu}^\varepsilon)e_{i||j}^\var(\bv^\varepsilon) \sqrt{g^\varepsilon}  dx^\varepsilon,
\\
(A \bu^\var,\bv^\var)_{V}&:=\int_{\Omega^\varepsilon}A^{ijkl,\varepsilon}e^\varepsilon_{k||l}(\bu^\varepsilon)e^\varepsilon_{i||j}(\bv^\varepsilon)\sqrt{g^\varepsilon} dx^\varepsilon, 
\\
(\bbf,\bv^\var)_{V}&:=\int_{\Omega^\varepsilon} f^{i,\varepsilon} v_i^\varepsilon \sqrt{g^\varepsilon} dx^\varepsilon + \int_{\Gamma_1^\varepsilon} h^{i,\varepsilon} v_i^\varepsilon\sqrt{g^\varepsilon}  d\Gamma^\varepsilon ,
\end{align*}
for all $\bu^\var, \bv^\var \in V$. The operators $B$ and $A$ are strongly monotone  as a consequence of the ellipticity of the fourth order tensors $(A^{ijkl,\var})$ and $(B^{ijkl,\var})$ in (\ref{elipticidadA})--(\ref{elipticidadB}).
Hence, the  Problem \ref{problema_eps} can be written as :
\begin{problem}
Find $\bu^\var:[0,T]\times\Omega^\var\longrightarrow \mathbb{R}^3$ such that,
\begin{align*} 
& \bu^\var(t)\in V \forallt,\\
& B \dot{\bu}^\var(t) + A \bu^\var(t)=\bbf(t) \ae,\\ 
&\bu^\var(0)=\bu_0^\var \ \textrm{in}\ V.
\end{align*}
\end{problem}
 Therefore,  we can apply Theorem \ref{teorema_existenciayunicidad} and conclude that $\bu^\var\in{W}^{1,2}(0,T;V)$. Moreover, if  $\dot{\fb}^{i,\var}\in L^{2}(0,T; L^2(\Omega^\var)) $, $\dot{\bh}^{i,\var}\in L^{2}(0,T; L^2(\Gamma_1^\var))$,  then we are in conditions of the Corollary \ref{Cor_ex_un_reg} and we conclude that $\bu^\var\in W^{2,2}(0,T;V)$.
\end{proof}

 \begin{theorem}\label{Theorema_exist_escalado_sin_orden} 
  Let $\Omega$ be a domain in $\mathbb{R}^3$ defined as in Section \ref{seccion_dominio_ind} and let $\bTheta$ be a $\mathcal{C}^2$-diffeomorphism of $\bar{\Omega}$ onto its image $\bTheta(\bar{\Omega})$, such that the three vectors $\bg_i=\d_i\bTheta(\bx)$ are linearly independent for all $\bx\in\bar{\Omega}$. Let $\fb^{i}(\var)\in L^{2}(0,T; L^2(\Omega)) $, $\bh^{i}(\var)\in L^{2}(0,T; L^2(\Gamma_1))$, where $\Gamma_1:= \Gamma_+\cup\Gamma_-$. Let  $\bu_0(\var)\in V(\Omega). $ Then, there exists a unique solution $\bu(\var)=(u_i(\var)):[0,T]\times\Omega \rightarrow \mathbb{R}^3$ satisfying the Problem \ref{problema_escalado}. Moreover $\bu(\var)\in W^{1,2}(0,T;V(\Omega))$. In addition to that, if $\dot{\fb}^{i}(\var)\in L^{2}(0,T; L^2(\Omega)) $, $\dot{\bh}^{i}(\var)\in L^{2}(0,T; L^2(\Gamma_1))$, then $\bu(\var)\in W^{2,2}(0,T;V(\Omega))$.
  \end{theorem}
  \begin{proof}
  The proof of this theorem is analogous to the proof in Theorem \ref{Thexistunic}, taking into account the ellipticity of the scaled fourth-order tensors in (\ref{elipticidadA_eps})--(\ref{elipticidadB_eps})
  and applying a corollary of Theorem \ref{teorema_existenciayunicidad} with $V=V(\Omega)$. Moreover, if  $\dot{\fb}^{i}(\var)\in L^{2}(0,T; L^2(\Omega)) $, $\dot{\bh}^{i}(\var)\in L^{2}(0,T; L^2(\Gamma_1))$,  then we are in conditions of the Corollary \ref{Cor_ex_un_reg} and we conclude that $\bu(\var)\in W^{2,2}(0,T;V(\Omega))$.
  \end{proof}

Now, let $\tilde{V}:= W^{1,2}(0,T; \calQ)$, where $\mathcal{Q}:=\{ (\Phi_{\alpha\beta})\in \mathbb{S}^2 ; \Phi_{\alpha\beta}\in L^2(\omega)\}$. Notice that $\left(\calQ, (\cdot, \cdot) \right)$ is a Hilbert space, where $ (\cdot, \cdot)  $ denotes its inner product.  We define the operators  $a: \calQ\times \calQ \longrightarrow \mathbb{R}$,  $b: \calQ\times \calQ \longrightarrow \mathbb{R}$ and  $c: \calQ\times \calQ \longrightarrow \mathbb{R}$ by
\begin{align} \label{operador_a}
a(\Sigma, \Phi):= \into \a \Sigma_{\sigma \tau} \Phi_{\alpha \beta} \sqrt{a} dy,
\\ \label{operador_b}
b(\Sigma, \Phi):= \into \b \Sigma_{\sigma \tau} \Phi_{\alpha \beta} \sqrt{a} dy,
\\ \label{operador_c}
c(\Sigma, \Phi):= \into \c \Sigma_{\sigma \tau} \Phi_{\alpha \beta} \sqrt{a} dy,
\end{align}
for all $ \Sigma, \Phi \in\calQ, $ where $\a, \b$ and $\c$ denote the contravariant components of three fourth order two-dimensional elliptic tensors. 

\begin{theorem} \label{teorema_existencia_bidimensional}
Let $f\in L^p( 0, T ;\calQ)$ with $p \geq 2$ , $\Sigma_0 \in \calQ$ and a constant $k>0$. Consider the strongly monotone, Lipschitz-continuous  operators $a,b,c :\calQ \times \calQ \longrightarrow \mathbb{R}$ defined  in (\ref{operador_a})--(\ref{operador_c}). Then, there exists   $\Sigma:[0,T]\longrightarrow\calQ $ unique solution to the problem
\begin{align} \label{ecuacion_operadores}
&a(\Sigma, \Phi) + b(\dot{\Sigma}, \Phi) - c\left( \int_0^t e^{-k(t-s)} \Sigma(s) ds, \Phi\right) = \left( f(t), \Phi\right), \ \forall \Phi\in\calQ, \aes,
\\ \label{condicion_operadores}
&\Sigma(0)=\Sigma_0.
\end{align}
Moreover, $\Sigma\in \tilde{V}$.  In addition, if  $\dot{\fb}\in L^{2}(0,T; \calQ) $,  then $\Sigma\in W^{2,2}(0,T;\calQ)$.
\end{theorem}
\begin{proof}
We first consider the auxiliary problem
\begin{align} \label{ecuacion_auxiliar}
&a(\Sigma_\theta, \Phi) + b(\dot{\Sigma}_\theta, \Phi)  = \left( f(t), \Phi\right) + c\left( \theta, \Phi\right), \ \forall \Phi\in\calQ \aes,
\\ \label{condicion_auxiliar}
&\Sigma_\theta(0)=\Sigma_0,
\end{align}
where $\theta\in \tilde{V}$. Notice that by the Riesz Representation Theorem we find that there exist bounded linear operators  $\tilde{B}:\calQ\longrightarrow \calQ ,$ $\tilde{A}:\calQ\longrightarrow \calQ$ and $\tilde{f}\in \calQ$ such that
\begin{align*}
(\tilde{B}\Sigma_\theta,\Phi)&:=b({\Sigma}_\theta, \Phi), \\
(\tilde{A }\Sigma_\theta,\Phi)&:= a(\Sigma_\theta, \Phi), \\
(\tilde{f},\Phi)&:=\left( f(t), \Phi\right) + c\left( \theta, \Phi\right),
\end{align*}
for all $\Sigma_\theta,\Phi \in \calQ$. Moreover, the operators $\tilde{A}$ and $\tilde{B}$ are strongly monotone by the definitions (\ref{operador_a})--(\ref{operador_b}). Therefore,  following similar arguments as in the proof of  Theorem \ref{teorema_existenciayunicidad}, we conclude that there exists a unique solution of the auxiliary problem satisfying $\Sigma_\theta\in \tilde{V}$. Now, we consider the operator $\Psi: \tilde{V}\longrightarrow \tilde{V}$ given by,
\begin{align*}
\Psi\theta(t)= \int_0^t e^{-k(t-s)}\Sigma_\theta(s) ds,
\end{align*}     
where $\Sigma_\theta$ is the solution of (\ref{ecuacion_auxiliar})--(\ref{condicion_auxiliar}). Let $\theta_1,\theta_2, \Sigma_{\theta_1}, \Sigma_{\theta_2} \in \tilde{V}$, hence by (\ref{ecuacion_auxiliar}) we can find that,
\begin{align*}
a(\Sigma_{\theta_1}- \Sigma_{\theta_2}, \Sigma_{\theta_1}- \Sigma_{\theta_2}) + \frac{1}{2}\frac{\d}{\d t}\left(b({\Sigma}_{\theta_1} - {\Sigma}_{\theta_2}, \Sigma_{\theta_1}- \Sigma_{\theta_2}) \right) =  - c\left( \theta_1 - \theta_2,  \Sigma_{\theta_2}- \Sigma_{\theta_1}\right).
\end{align*}
Since the operator $a$ is  strongly monotone we find that,
\begin{align*}
\frac{1}{2}\frac{\d}{\d t}\left(b({\Sigma}_{\theta_1} - {\Sigma}_{\theta_2}, \Sigma_{\theta_1}- \Sigma_{\theta_2}) \right) \leq  - c\left( \theta_1 - \theta_2,  \Sigma_{\theta_2}- \Sigma_{\theta_1}\right).
\end{align*}
Integrating with respect to the time variable we find that,
\begin{align}\label{refA}
b({\Sigma}_{\theta_1} - {\Sigma}_{\theta_2}, \Sigma_{\theta_1}- \Sigma_{\theta_2})  \leq  - \int_0^t c\left( \theta_1 - \theta_2,  \Sigma_{\theta_2}- \Sigma_{\theta_1}\right) ds.
\end{align}
In what follows let $||\cdot||$ denote a norm induced by the inner product in $\calQ$. Moreover, by the continuity of the operator $c$ , there exists a constant $c_1>0$ such that
\begin{align}\nonumber
 &- \int_0^t c\left( \theta_1 - \theta_2,  \Sigma_{\theta_2}- \Sigma_{\theta_1}\right) ds \leq ||   \int_0^t c\left( \theta_1 - \theta_2,  \Sigma_{\theta_2}- \Sigma_{\theta_1}\right) ds|| 
 \\ \nonumber
 \qquad &\leq   \int_0^t ||c\left( \theta_1 - \theta_2,  \Sigma_{\theta_2}- \Sigma_{\theta_1}\right)|| ds \leq c_1 \int_0^t || \theta_1 - \theta_2|| || \Sigma_{\theta_2}- \Sigma_{\theta_1}|| ds
\\ \label{refA2}
& \qquad \leq \frac{c_1}{2} \int_0^t\left( || \theta_1 - \theta_2||^2 + || \Sigma_{\theta_2}- \Sigma_{\theta_1}||^2 \right)ds.
\end{align} 
On the other hand, since $b$ is a strongly monotone operator, there exists a constant $c_2>0$ such that
\begin{align*}
\frac{1}{2}b({\Sigma}_{\theta_1} - {\Sigma}_{\theta_2}, \Sigma_{\theta_1}- \Sigma_{\theta_2}) \geq c_2 ||{\Sigma}_{\theta_1} - {\Sigma}_{\theta_2}||^2, 
\end{align*}
hence, together with (\ref{refA})--(\ref{refA2}) we obtain the following inequality,
\begin{align*}
c_2 ||{\Sigma}_{\theta_1} - {\Sigma}_{\theta_2}||^2 \leq \frac{c_1}{2} \int_0^t || \theta_1 - \theta_2||^2ds + \frac{c_1}{2} \int_0^t|| \Sigma_{\theta_2}- \Sigma_{\theta_1}||^2 ds.
\end{align*}
Applying Gronwall's inequality we find that there exists a $C>0$ such that
\begin{align*}
||{\Sigma}_{\theta_1}(t) - {\Sigma}_{\theta_2}(t)||^2 \leq C  \int_0^t || \theta_1(s) - \theta_2(s)||^2ds.
\end{align*}
for all $t\in[0,T].$ Therefore,
\begin{align*}
||\Psi\theta_1(t) - \Psi\theta_2(t)||^2 \leq C  \int_0^t || \theta_1(s) - \theta_2(s)||^2ds.
\end{align*}
for all $t\in[0,T].$ Furthermore,
\begin{align*}
\frac{\d}{\d t}\left( \Psi\theta(t)\right)= \Sigma_\theta(t)- k \int_0^t e^{-k(t-s)}\Sigma_\theta(s)ds.
\end{align*} 
As a consequence, there exists a $n\in \mathbb{N}$ such that $||\Psi^n\theta_1- \Psi^n\theta_2||_{\tilde{V}}< || \theta_1-\theta_2||_{\tilde{V}}$. By the Banach fixed point theorem, there exists a unique $\theta^*$ such that $\Psi\theta^*(t)=\theta^*(t)$, $\forall t \in [0,T]$. Hence, the auxiliary problem  (\ref{ecuacion_auxiliar})--(\ref{condicion_auxiliar}) for $\theta=\theta^*$ is a reformulation of the original problem (\ref{ecuacion_operadores})--(\ref{condicion_operadores}). Therefore, there exists a unique solution of the original problem  satisfying $\Sigma\in\tilde{V}$. Moreover, if  $\dot{\fb}\in L^{2}(0,T; \calQ) $,  applying a modified version of the arguments in Corollary \ref{Cor_ex_un_reg} we conclude that $\Sigma\in W^{2,2}(0,T;\calQ)$.
\end{proof}

\section{Formal Asymptotic Analysis} \setcounter{equation}{0} \label{procedure}

In this section, we highlight some relevant steps in the construction of the formal asymptotic expansion of the scaled unknown variable $\bu(\var)$ including the characterization of the zeroth-order term, and the derivation of some key results which will lead to the two-dimensional equations of the viscoelastic shell problems. We define the scaled applied forces as,
 \begin{align*} 
 & \bbf(\varepsilon)(t, \bx)=\varepsilon^p\bbf^p(t,\bx) \ \forall \bx\in \Omega \ \textrm{and} \ \forall t\in[0,T], \\ 
 & \bbh(\varepsilon)(t, \bx)=\varepsilon^{p+1}\bbh^{p+1}(t,\bx) \ \forall  \bx\in \Gamma_+\cup\Gamma_- \ \textrm{and} \ \forall t\in[0,T],
 \end{align*}
where $p$ is a natural number that will show the order of the volume and surface forces, respectively. We substitute in  (\ref{ec_problema_escalado}) to obtain the following problem:
 \begin{problem}\label{problema_orden_fuerzas}
  Find $\bu(\varepsilon):[0,T]\times\Omega\longrightarrow \mathbb{R}^3$ such that,
  \begin{align} \nonumber
  & \bu(\varepsilon)(t,\cdot)\in V(\Omega) \forallt, \\ \nonumber
      &\int_{\Omega}A^{ijkl}(\varepsilon)e_{k||l}(\varepsilon;\bu(\varepsilon))e_{i||j}(\varepsilon;\bv)\sqrt{g(\varepsilon)} dx
   + \int_{\Omega} B^{ijkl}(\varepsilon)e_{k||l}(\varepsilon;\dot{\bu}(\varepsilon))e_{i||j}(\varepsilon;\bv) \sqrt{g(\varepsilon)}  dx
     \\\label{ecuacion_orden_fuerzas}
     &\quad= \int_{\Omega} \var^p \fb^{i,p}v_i \sqrt{g(\varepsilon)} dx + \int_{\Gamma_+\cup\Gamma_-} \var^{p}\bh^{i,p+1} v_i\sqrt{g(\varepsilon)}  d\Gamma  \quad \forall \bv\in V(\Omega), \aes,
     \\\displaystyle \nonumber
     & \bu(\var)(0,\cdot)= \bu_0(\var)(\cdot).
    \end{align} 
  \end{problem}
  
\begin{remark}
The existence and uniqueness of solution of Problem \ref{problema_orden_fuerzas} follows using analogous arguments as in  Theorem \ref{Theorema_exist_escalado_sin_orden}.
\end{remark}


 Assume that $\btheta\in\mathcal{C}^3(\bar{\omega};\mathbb{R}^3)$ and that the scaled unknown $\bu(\varepsilon)$ and scaled initial displacement $\bu_0(\var)$ admit an asymptotic expansion of the form
\begin{align}\label{desarrollo_asintotico}
\bu(\varepsilon)&= \bu^0 + \varepsilon \bu^1 + \varepsilon^2 \bu^2 +... \quad \textrm{with} \ \bu^0\neq \mathbf{0},  \\ \nonumber
\bu_0(\varepsilon)&= \bu_0^0 + \varepsilon \bu_0^1 + \varepsilon^2 \bu_0^2 +... . \quad \textrm{with} \ \bu^0_0=\bu^0(0,\cdot),
\end{align}
where $\bu^0(t)\in V(\Omega),$ $ \bu^q(t)\in [H^1(\Omega)]^3 \ae$ and $\bu^0_0\in V(\Omega),$ $ \bu^q_0\in [H^1(\Omega)]^3$ with $q\geq1$. The assumption (\ref{desarrollo_asintotico}) implies an asymptotic expansion of the scaled linear strain as follows
\begin{align*}
\eij(\var)\equiv\eij(\varepsilon;\bu(\varepsilon))&=\frac1{\varepsilon}\eij^{-1}+ \eij^0 + \varepsilon\eij^1 + \varepsilon^2\eij^2 + \varepsilon^3\eij^3+...
\end{align*}
where,

\begin{align*}
\left\{\begin{aligned}[c]
\eab^{-1}&=0, \\
  \eatres^{-1}&=\frac{1}{2}\d_3u_\alpha^0, 
   \\
  \edtres^{-1}&=\d_3u_3^0,
\end{aligned}\right.
\qquad \qquad \qquad
\left\{\begin{aligned}[c] 
 \eab^0&=\frac{1}{2}(\d_\beta u_\alpha^0 + \d_\alpha u_\beta^0) - \Gamma_{\alpha\beta}^\sigma u_\sigma^0 - b_{\alpha\beta}u_3^0, 
 \\
\eatres^0&=\frac{1}{2}(\d_3 u_\alpha^1 + \d_\alpha u_3^0) +   b_{\alpha}^\sigma u_\sigma^0, 
\\
\edtres^0&=\d_3u_3^1,
\end{aligned}\right.\qquad
\end{align*}
\begin{equation}\label{eij_terminos_expansion_u}
\end{equation}
\begin{align*}
\left\{\begin{aligned}[c]
\eab^1&=\frac{1}{2}(\d_\beta u_\alpha^1 + \d_\alpha u_\beta^1) - \Gamma_{\alpha\beta}^\sigma u_\sigma^1 - b_{\alpha\beta}u_3^1 + x_3(b_{\beta|\alpha}^\sigma u_\sigma^0  + b_\alpha^\sigma b_{\sigma\beta}u_3^0), \\
\eatres^1&=\frac{1}{2}(\d_3 u_\alpha^2 + \d_\alpha u_3^1) +   b_{\alpha}^\sigma u_\sigma^1 + x_3b_\alpha^\tau  b_\tau^\sigma u_\sigma^0, \\
\edtres^1&=\d_3 u_3^2.
\end{aligned}\right. \qquad \quad \qquad 
\end{align*}

In addition, the functions $\eij(\varepsilon;\bv) $ admit the following expansion,
\begin{align*}
\eij(\varepsilon;\bv)=\frac{1}{\varepsilon}\eij^{-1}(\bv) + \eij^0(\bv) + \varepsilon\eij^1(\bv)+...
\end{align*}
where,
\begin{align*}
\left\{\begin{aligned}[c]
\eab^{-1}(\bv)&=0,\\
 \eatres^{-1}(\bv)&=\frac{1}{2}\d_3v_\alpha, 
   \\
\edtres^{-1}(\bv)&=\d_3v_3,
\end{aligned}\right.
\qquad \qquad \quad
\left\{\begin{aligned}[c] 
 \eab^0(\bv)&=\frac{1}{2}(\d_\beta v_\alpha + \d_\alpha v_\beta) - \Gamma_{\alpha\beta}^\sigma v_\sigma - b_{\alpha\beta}v_3, 
 \\
\eatres^0(\bv)&=\frac{1}{2} \d_\alpha v_3 +   b_{\alpha}^\sigma v_\sigma, 
\\
\edtres^0(\bv)&=0,
\end{aligned}\right.
\end{align*}
\begin{equation}\label{eij_terminos_expansion}
\end{equation}
\begin{align*}
\left\{\begin{aligned}[c]
\eab^1(\bv)&=  x_3b_{\beta|\alpha}^\sigma v_\sigma  + x_3b_\alpha^\sigma b_{\sigma\beta}v_3, \\\nonumber
\eatres^1(\bv)&= x_3b_\alpha^\tau b_\tau^\sigma v_\sigma, \\\nonumber
\edtres^1(\bv)&=0.
\end{aligned}\right. \qquad \quad \qquad \qquad \qquad \qquad \qquad \qquad \qquad \qquad 
\end{align*}

Upon substitution on (\ref{ecuacion_orden_fuerzas}), we proceed to characterize the different terms involved in the asymptotic expansions considering different values for $p$, that is, taking different orders for the applied forces. Assume that
\begin{equation}\label{condicion_inicial_indep_3}
\d_3 \bu_0^0=\bcero,
\end{equation} this is, that the zeroth-order term of the initial displacement is independent of the transversal variable. Also,  we assume that the initial condition for the scaled linear strains is such that
\begin{equation} \label{condicion_inicial_def}
\eij^0(0,\cdot)=\eij^1(0,\cdot)=0,
\end{equation}
this is, the strains  at the beginning of the period of observation are of order $O(\var^2)$ at least (since by (\ref{eij_terminos_expansion_u}) and (\ref{condicion_inicial_indep_3})  we have that $\eij^{-1}(0,\cdot)=0$).


We shall now identify the leading term $\bu^0$ of the expansion (\ref{desarrollo_asintotico}) by canceling the other terms of the successive powers of $\var$ in the equations of the Problem \ref{problema_orden_fuerzas}. We will show that $\bu^0$ is solution of a two-dimensional problem of a viscoelastic membrane or flexural shell depending on several factors, and that the orders of applied forces are determined in both cases. Given $\beeta=(\eta_i)\in [H^1(\omega)]^3,$ let
 \begin{equation} \label{def_gab}
 \gab(\beeta):= \frac{1}{2}(\d_\beta\eta_\alpha + \d_\alpha\eta_\beta) - \Gamma_{\alpha\beta}^\sigma\eta_\sigma -  b_{\alpha\beta}\eta_3,
 \end{equation}
 denote the covariant components of the linearized change of metric tensor associated with a displacement field $\eta_i\ba^i$ of the surface $S$. Let us define the spaces, 
 \begin{align}\nonumber
 V(\omega)&:=\{\beeta=(\eta_i)\in[H^1(\omega)]^3 ; \eta_i=0 \ \textrm{on} \ \gamma_0 \}, \\ \nonumber
 V_0(\omega)&:=\{\beeta=(\eta_i)\in V(\omega), \gamma_{\alpha\beta}(\beeta)=0  \ \textrm{in} \ \omega \}, \\ \nonumber
 V_F(\omega)&:= \{ \beeta=(\eta_i) \in H^1(\omega)\times H^1(\omega)\times H^2(\omega) ; \eta_i=\d_\nu \eta_3=0 \ \textrm{on} \ \gamma_0, \gab(\beeta)=0 \en \omega \}.
 \end{align}

\begin{theorem}
Consider the Problem \ref{problema_orden_fuerzas} upon substitution of the expansion for $\bu(\var)$ proposed in (\ref{desarrollo_asintotico}). Identifying the terms multiplied by the same powers of $\var$ we find that:

\begin{enumerate}[label={{(\roman*)}}, leftmargin=0em ,itemindent=3em]
\item The main leading term $\bu^0$ of the asymptotic expansion is independent of the transversal variable $x_3$. Therefore, it can be identified with a function $\bxi^0\in [H^1(\omega)]^3$ such that $\bxi^0=\bcero$ on $\gamma_0$ and also we can identify $\bu_0^0$ with a function $\bxi^0_0(\cdot)=\bxi^0(0,\cdot)$. As a consequence,
\begin{equation*}
\eij^{-1}(t)=0 \en \Omega, \forallt.
\end{equation*}
\item The following zeroth-order terms of the scaled linearized strains are identified. On one hand,
\begin{equation*}
\eatres^0(t)=0 \en \Omega, \forallt.
\end{equation*}
On the other hand, if we assume  $\theta>0$ we obtain that 
\begin{align}\label{edtres_cero}
\edtres^0(t)= - \frac{\theta}{\theta + \rho} \left( a^{\alpha \beta }\eab^0(t) + \Lambda\int_0^te^{-k(t-s)}a^{\alpha\beta}\eab^0(s) ds \right), \en \Omega, \forallt,
\end{align}
  where,
\begin{align}\label{Constantes}
\Lambda:=\left(  \frac{\lambda}{\theta} - \frac{\lambda+ 2 \mu}{\theta + \rho}  \right), \quad k:=\frac{\lambda+ 2 \mu}{\theta + \rho} .
\end{align}
Moreover,
\begin{align*}
\dedtres^0(t)= - \frac{\lambda}{\theta + \rho} a^{\alpha \beta} \eab^0(t)- \frac{\lambda + 2\mu}{\theta + \rho} \edtres^0(t) - \frac{\theta}{\theta + \rho}  a^{\alpha\beta}\deab^0(t),
\end{align*}
 $\en \Omega \ae$.

\item  The following equality is verified,
\begin{align*}
&\frac{1}{2}\intO \a \est^0\eab^0(\beeta)\sqrt{a}dx + \frac{1}{2}\intO \b \dest^0 \eab^0(\beeta)\sqrt{a}dx 
 \\
 & \qquad- \frac{1}{2}\int_0^te^{-k(t-s)}\intO \c \est^0(s)\eab^0(\beeta)\sqrt{a}dx ds 
 \\ & \quad = \intO f^{i,0}\eta_i \sqrt{a}dx + \intG h^{i,1} \eta_i \sqrt{a} d \Gamma, \ \forall\beeta\in V(\omega) \aes,
\end{align*}
where  $\a$ , $\b$ and $\c$ denote the contravariant components of the   fourth order two-dimensional tensors,  defined as follows: 
 \begin{align} \label{tensor_a_bidimensional}
   \a&:=\frac{2\lambda\rho^2 + 4\mu\theta^2}{(\theta + \rho)^2}a^{\alpha\beta}a^{\sigma\tau} + 2\mu \ten, 
  \\ \label{tensor_b_bidimensional}
    \b&:=\frac{2\theta\rho}{\theta + \rho}a^{\alpha\beta}a^{\sigma\tau} + \rho\ten, 
     \\ \label{tensor_c_bidimensional}
   \c&:=\frac{2 \left(\theta \Lambda \right)^2}{\theta + \rho} a^{\alpha\beta}a^{\sigma\tau}.  
 \end{align}

 Moreover, 
\begin{align} \label{et3}
\eab^0(t)=\gab(\bxi^0(t)) \ \textrm{and} \ \eab^0(\beeta(t))=\gab(\beeta(t)) \ \textrm{for all}\ \beeta\in V(\omega) \forallt.
\end{align}
\item Assume that  $V_0(\omega)=\{\bcero\}$. Then we have that  $\bxi^0$ is solution of the two-dimensional limit equations, known as the viscoelastic membrane shell equations: Find $\bxi^0:[0,T] \times\omega \longrightarrow \mathbb{R}^3$ such that,
    \begin{align*}\nonumber 
    & \bxi^0(t,\cdot)\in V(\omega) \forallt,\\ \nonumber
   &\int_{\omega} \a\gst(\bxi^0)\gab(\beeta)\sqrt{a}dy +\int_{\omega}\b\gst(\dot{\bxi}^0)\gab(\beeta)\sqrt{a}dy
   \\ 
   &- \int_0^te^{-k(t-s)}\into \c \gst(\bxi^0(s))\gab(\beeta)\sqrt{a}dyds 
   \\ 
   &\quad=\int_{\omega}p^{i,0}\eta_i\sqrt{a}dy \ \forall \beeta=(\eta_i)\in V(\omega), \aes,
    \\
    &\bxi^0(0,\cdot)=\bxi^0_0(\cdot),
   \end{align*} 
   where,
   \begin{align}\label{p0}
   p^{i,0}(t):=\int_{-1}^{1}\fb^{i,0}(t)dx_3+h_+^{i,1}(t)+h_-^{i,1}(t) \ \textrm{and} \ h_{\pm}^{i,1}(t)=\bh^{i,1}(t,\cdot,\pm 1) \forallt.
   \end{align}

\item Assume that $ V_0(\omega)\neq\{\bcero\} $. We find that
\begin{align*}
\eij^0(t)&=0 \en \Omega, \Forallt,\\
\bxi^0(t)&\in V_F(\omega) \Forallt.
\end{align*}
Moreover, assume that $\bu^1(t)\in V(\Omega) \Forallt$. Then, there exists a function $\bxi^1(t)=(\xi_i^1(t))\in V(\omega) \Forallt$, such that 
\begin{align*}
u_\alpha^1(t)&=\xi_\alpha^1(t) - x_3( \d_\alpha\xi_3^0(t) + 2 b_\alpha^\sigma \xi_\sigma^0(t)),\\
u_3^1(t)&=\xi_3^1(t),
\end{align*}
$\Forallt.$ Also, the following first-order terms of the scaled linearized strains  are identified. On one hand,
\begin{align*}
\eatres^1(t)=0 \en \Omega, \Forallt.
\end{align*}
On the other hand, we obtain that,
\begin{align*}
\edtres^1(t)= - \frac{\theta}{\theta + \rho} \left( a^{\alpha \beta }\eab^1(t) + \Lambda\int_0^te^{-k(t-s)}a^{\alpha\beta}\eab^1(s) ds \right), \en \Omega, \Forallt,
\end{align*}
 and where $\Lambda$ and $k$ are defined as in (\ref{Constantes}). Moreover,
\begin{align*}
\dedtres^1(t)= - \frac{\lambda}{\theta + \rho} a^{\alpha \beta} \eab^1(t)- \frac{\lambda + 2\mu}{\theta + \rho} \edtres^1(t) - \frac{\theta}{\theta + \rho}  a^{\alpha\beta}\deab^1(t),
\end{align*}
 $\en \Omega, \ae$.
Furthermore, let
\begin{equation} \label{rab}
\rho_{\alpha\beta}(\beeta):= \d_{\alpha\beta}\eta_3 - \Gamma_{\alpha\beta}^\sigma \d_\sigma\eta_3 - b_\alpha^\sigma b_{\sigma\beta} \eta_3 + b_\alpha^\sigma (\d_\beta\eta_\sigma- \Gamma_{\beta\sigma}^\tau \eta_\tau) + b_\beta^\tau(\d_\alpha\eta_\tau-\Gamma_{\alpha\tau}^\sigma\eta_\sigma ) + b^\tau_{\beta|\alpha} \eta_\tau,
\end{equation}
denote the covariant components of the linearized change of curvature tensor associated with a displacement  field $\eta_i \ba^i$ of the surface $S$. Then
\begin{equation}\label{ref1}
\eab^1(t)=\gab(\bxi^1(t))- x_3\rab(\bxi^0(t)) \Forallt.
\end{equation}

\item Assume that $ V_0(\omega)\neq\{\bcero\} $, then
\begin{align*}
\bxi^1(t) \in V_0(\omega) \Forallt.
\end{align*}

\item For the case where $V_0(\omega)\neq\{\bcero\}$, we find  that $\bxi^0$ is solution of the  two-dimensional limit  equations known as viscoelastic flexural shell equations: Find $\bxi^0:(0,T) \times\omega \longrightarrow \mathbb{R}^3$ such that,
    \begin{align*}\nonumber
    & \bxi^0(t,\cdot)\in V_F(\omega) \forallt,\\ \nonumber
   &\frac{1}{3}\int_{\omega} \a\rst(\bxi^0)\rab(\beeta)\sqrt{a}dy +\frac{1}{3}\int_{\omega}\b\rst(\dot{\bxi}^0)\rab(\beeta)\sqrt{a}dy
   \\ \nonumber
   &- \frac{1}{3}\int_0^te^{-k(t-s)}\into \c \rst(\bxi^0(s))\rab(\beeta)\sqrt{a}dyds
   \\ 
   &\quad=\int_{\omega}p^{i,2}\eta_i\sqrt{a}dy \ \forall \beeta=(\eta_i)\in V_F(\omega), \aes,
    \\
    &\bxi^0(0,\cdot)=\bxi^0_0(\cdot),
   \end{align*} 

   where,
   \begin{align} \label{p2}
 p^{i,2}(t):=\int_{-1}^{1}\fb^{i,2}(t)dx_3+h_+^{i,3}(t)+h_-^{i,3}(t) \ \textrm{and} \ h_{\pm}^{i,3}(t)=\bh^{i,3}(t,\cdot,\pm 1) \forallt.
   \end{align}
\end{enumerate}
\end{theorem}

\begin{proof} For the proof of this theorem firstly, we will take values for $p$ on the Problem \ref{problema_orden_fuerzas}. Then, we group terms multiplied by the same powers of $ \var$, canceling the terms of the expansion proposed.

\begin{enumerate}[label={{(\roman*)}}, leftmargin=0em ,itemindent=3em] 
\item Let $p=-2$ in (\ref{ecuacion_orden_fuerzas}). Hence, grouping the terms multiplied by $\var^{-2}$ (see (\ref{tensorA_tildes})--(\ref{tensorB_tildes})) we find that 
\begin{align}\nonumber
&\int_{\Omega} A^{ijkl}(0)\ekl^{-1}\eij^{-1}(\bv) \sqrt{a}dx + \int_{\Omega} B^{ijkl}(0)\dekl^{-1}\eij^{-1}(\bv) \sqrt{a}dx
\\ \label{et2}
&\quad=\int_{\Omega } f^{i,-2} v_i \sqrt{a} dx + \intG h^{i,-1} v_i \sqrt{a}d\Gamma.
\end{align}
Considering $\bv\in V(\Omega)$ independent of $x_3$ (see  (\ref{eij_terminos_expansion})),  the left-hand side of the  equation (\ref{et2}) cancels. Hence, in order to avoid compatibility conditions between the applied forces we must take $f^{i,-2}=0$ and $h^{i,-1}=0$. So that, back on  the equation (\ref{et2}), using (\ref{eij_terminos_expansion_u}), (\ref{eij_terminos_expansion}) and Theorem \ref{Th_comportamiento asintotico}, leads to
\begin{align}\nonumber
&\intO A^{ijkl}(0)\ekl^{-1}\eij^{-1}(\bv)\sqrt{a} dx + \intO B^{ijkl}(0)\dekl^{-1}\eij^{-1}(\bv)\sqrt{a}dx
\\\nonumber
&\quad = \intO\left(4A^{\alpha 3 \sigma 3}(0)\estres^{-1}\eatres^{-1}(\bv) + A^{3333}(0)\edtres^{-1} \edtres^{-1}(\bv)  \right) \sqrt{a}dx
\\\nonumber
& \qquad +\intO\left(4B^{\alpha 3 \sigma 3}(0)\destres^{-1}\eatres^{-1}(\bv) + B^{3333}(0)\dedtres^{-1} \edtres^{-1}(\bv)  \right) \sqrt{a}dx
\\\nonumber
&\quad = \intO \left( \mu a^{\alpha\sigma} \d_3u_\sigma^0 \d_3v_\alpha + (\lambda + 2\mu) \d_3u_3^0 \d_3 v_3 \right) \sqrt{a} dx
\\ \label{ecuacion_int_ui}
& \qquad + \intO \left( \frac{\rho}{2} a^{\alpha\sigma} \d_3\dot{u}_\sigma^0 \d_3v_\alpha + (\theta + \rho) \d_3\dot{u}_3^0 \d_3 v_3 \right) \sqrt{a} dx=0,
\end{align}
for all $\bv=(v_i)\in V(\Omega), \aes$. Let $\bv=(v_i)\in V(\Omega)$ such that $v_\alpha=0$. By the Theorem \ref{th_int_nula}, we obtain the following differential equation
\begin{equation*}
(\lambda + 2 \mu) \d_3 u_3^0 + (\theta + \rho) \d_3 \dot{u}_3^0=0.
\end{equation*}
This equation together with the initial condition $(\ref{condicion_inicial_indep_3})$, leads to
\begin{equation*}
\d_3u_3^0(t)=0 \en \Omega, \ \textrm{for all} \ t\in[0,T].
\end{equation*}
Now, taking $v_\alpha=u_\alpha^0$ in  (\ref{ecuacion_int_ui}), we have
\begin{align*}
\intO \mu a^{\alpha\sigma} \d_3u_\sigma^0 \d_3 u_\alpha^0 \sqrt{a} dx + \intO \frac{\rho}{2} a^{\alpha\sigma} \d_3 \dot{u}_\sigma^0\d_3u_\alpha^0 \sqrt{a}dx=0, \ \textrm{a.e in} \ (0,T),
\end{align*}
that is equivalent to
\begin{align*}
\intO \mu a^{\alpha\sigma} \d_3u_\sigma^0 \d_3 u_\alpha^0 \sqrt{a} dx + \frac{\d}{\d t}\intO \frac{\rho}{4} a^{\alpha\sigma} \d_3 {u}_\sigma^0\d_3u_\alpha^0 \sqrt{a}dx=0, \ \textrm{a.e in} \ (0,T).
\end{align*}
Since the matrix $\left(a^{\alpha\sigma}\right)$ is positive definite, we have
\begin{align*}
 \frac{\d}{\d t}\intO \frac{\rho}{4} a^{\alpha\sigma} \d_3 {u}_\sigma^0\d_3u_\alpha^0 \sqrt{a}dx\leq0, \ \textrm{a.e in} \ (0,T).
\end{align*}
By integrating with respect to the time variable and by (\ref{condicion_inicial_indep_3}), we deduce
\begin{align*}
\intO  a^{\alpha\sigma} \d_3 {u}_\sigma^0\d_3u_\alpha^0 \sqrt{a}dx\leq0 \Forallt,
\end{align*}
and using again the positive definiteness of $\left(a^{\alpha\sigma}\right)$ we conclude 
\begin{equation*}
\d_3 u_\alpha^0(t)=0 \en \Omega,\Forallt.
\end{equation*}
Therefore, we have found that the main term $\bu^0$ of the asymptotic expansion is independent of the transversal variable$\Forallt$, hence, it can be identified with a function $\bxi^0(t)\in [H^1(\omega)]^3 \Forallt$ such that $\bxi^0=\bcero$ on $\gamma_0$, this is, $\bxi^0(t)\in V(\omega) \Forallt$. Moreover, as $\bu^0_0$ does not depend on $x_3$ as well by (\ref{condicion_inicial_indep_3}), we can identify $\bu_0^0$ with a  function $\bxi^0_0\in V(\omega)$ and it is verified that $\bxi^0_0(\cdot)=\bxi^0(0,\cdot)$.
Moreover, by (\ref{eij_terminos_expansion}) we obtain that
\begin{equation*}
\eij^{-1}(t)=0 \en \Omega, \Forallt.
\end{equation*}

\item Let now $p=-1$ in (\ref{ecuacion_orden_fuerzas}). Grouping the terms multiplied by $\var^{-1}$, we find (taking into account the results from the previous step $(i)$) that
\begin{align}\nonumber
&\intO A^{ijkl}(0)\ekl^0\eij^{-1}(\bv)\sqrt{a}dx + \intO B^{ijkl}(0)\dekl^0\eij^{-1}(\bv)\sqrt{a}dx 
\\ \label{ec_step2}
& \quad=\intO f^{i,-1}v_i \sqrt{a}dx + \intG h^{i,0}v_i \sqrt{a} d\Gamma,
\end{align}
for all $\bv\in V(\Omega), \aes$. Analogously to step $(i)$, considering a test function $\bv$ independent of $x_3$, we obtain  that $f^{i,-1}$ and $h^{i,0}$ must be zero. Therefore, from the left-hand side of the last equation we have
\begin{align} \nonumber
&\intO A^{ijkl}(0)\ekl^0\eij^{-1}(\bv)\sqrt{a}dx + \intO B^{ijkl}(0)\dekl^0\eij^{-1}(\bv)\sqrt{a}dx 
\\\nonumber
&\quad =\intO 4A^{\alpha 3 \sigma 3}(0)\eatres^0\estres^{-1}(\bv)\sqrt{a}dx + \intO\left( A^{\alpha\beta 33}(0)\eab^0 + A^{3333}(0)\edtres^0  \right)\edtres^{-1}(\bv) \sqrt{a}dx
\\\nonumber
& \qquad + \intO 4B^{\alpha 3 \sigma 3}(0)\deatres^0\estres^{-1}(\bv)\sqrt{a}dx + \intO\left( B^{\alpha\beta 33}(0)\deab^0 + B^{3333}(0)\dedtres^0  \right)\edtres^{-1}(\bv) \sqrt{a}dx
\\\nonumber 
&\quad= \intO \left(2\mu a^{\alpha \sigma}\eatres^0\d_3v_\sigma + \left(\lambda a^{\alpha\beta}\eab^0+ (\lambda + 2\mu)\edtres^0 \right)\d_3v_3\right)\sqrt{a}dx
\\ \label{ecuacion_integral1}
&\qquad+ \intO \left(\rho a^{\alpha \sigma}\deatres^0\d_3v_\sigma + \left(\theta a^{\alpha\beta}\deab^0+ (\theta + \rho)\dedtres^0 \right)\d_3v_3\right)\sqrt{a}dx=0.
\end{align}
On one hand, if we take $\bv\in V(\Omega)$ such that $v_2=v_3=0$ and using the Theorem \ref{th_int_nula}, we have
\begin{align} \label{ec_dif1}
2\mu a^{\alpha 1} \eatres^0 + \rho a^{\alpha 1} \deatres^0 =0 \aes.
\end{align}
On the other hand, if we take $\bv\in V(\Omega)$ such that $v_1=v_3=0$ and using the Theorem \ref{th_int_nula}, we have
\begin{align} \label{ec_dif2}
2\mu a^{\alpha 2} \eatres^0 + \rho a^{\alpha 2} \deatres^0 =0 \aes.
\end{align}

Multiplying (\ref{ec_dif1}) by $a^{22}$ and (\ref{ec_dif2}) by $-a^{21}$ and adding both expressions we have
\begin{align*}
2\mu \left(a^{22}a^{11} - a^{21}a^{12} \right) e_{1||3}^0 + \rho \left( a^{22}a^{11} - a^{21}a^{12}\right)\dot{e}_{1||3}^0 = 2\mu ae_{1||3}^0  + \rho a\dot{e}_{1||3}^0 =  0,
\end{align*}
$\aes$, by (\ref{definicion_a}). Now, by the initial condition in (\ref{condicion_inicial_def}) we conclude
\begin{align*}
e_{1||3}^0(t) =0 \en \Omega ,\Forallt.
\end{align*}
Multiplying (\ref{ec_dif1}) by $a^{12}$ and (\ref{ec_dif2}) by $-a^{11}$ and adding both expressions we have
\begin{align*}
2\mu ae_{2||3}^0  + \rho a\dot{e}_{2||3}^0 =  0 \aes,
\end{align*}
 Now, by the initial condition in (\ref{condicion_inicial_def}) we conclude
\begin{align*}
e_{2||3}^0(t) =0 \en \Omega ,\Forallt.
\end{align*}
Taking  $\bv\in V(\Omega)$ in (\ref{ecuacion_integral1}) such that $v_\alpha=0$, we obtain
\begin{align*} \nonumber
&\intO  \left(\lambda a^{\alpha\beta}\eab^0+ (\lambda + 2\mu)\edtres^0 \right)\d_3v_3\sqrt{a}dx
\\ 
&\qquad + \intO  \left(\theta a^{\alpha\beta}\deab^0+ (\theta + \rho)\dedtres^0 \right)\d_3v_3\sqrt{a}dx=0,
\end{align*}
for all $v_3\in H^1(\Omega)$ with $v_3=0 \en \Gamma_0, \aes$. By Theorem \ref{th_int_nula}, we obtain the following differential equation
\begin{align} \label{ecuacion_casuistica}
\lambda a^{\alpha\beta} \eab^0 + (\lambda+2\mu) \edtres^0 +\theta a^{\alpha \beta} \deab^0 + (\theta + \rho) \dedtres^0=0.
\end{align}
\begin{remark} \label{nota_desvio}
Note that removing time dependency and viscosity, that is taking $\theta=\rho=0$, the equation leads to the one studied in \cite{Ciarlet4b}, that is, the elastic case. 
\end{remark}
In order to solve the equation (\ref{ecuacion_casuistica}) in the more general case, we assume that the viscosity coefficient $\theta$ is strictly positive. Moreover, we can prove that this equation is equivalent to
\begin{align*} 
\theta e^{-\frac{\lambda}{\theta}t} \frac{\d}{\d t}\left(a^{\alpha\beta}\eab^0 e^{\frac{\lambda}{\theta}t}\right)=-\left(\theta + \rho \right)e^{-\frac{\lambda + 2\mu}{\theta + \rho}t} \frac{\d}{\d t}\left(\edtres^0e^{\frac{\lambda + 2\mu}{\theta + \rho}t}\right).
\end{align*}
Integrating with respect to the time variable and using (\ref{condicion_inicial_def}) we find that,
\begin{align*}
\edtres^0e^{\frac{\lambda + 2\mu}{\theta + \rho}t}= - \frac{\theta}{\theta + \rho} \int_0^t e^{\left(\frac{\lambda +2 \mu}{\theta + \rho} - \frac{\lambda}{\theta}\right) s}  \frac{\d}{\d s} \left(a^{\alpha\beta}\eab^0(s) e^{\frac{\lambda}{\theta}s}\right)ds,
\end{align*}
integrating by parts 
and simplifying we conclude that,
\begin{align*}
\edtres^0(t)= - \frac{\theta}{\theta + \rho} \left( a^{\alpha \beta }\eab^0(t) + \Lambda\int_0^te^{-k(t-s)}a^{\alpha\beta}\eab^0(s) ds \right),
\end{align*}
in $\Omega$,$\Forallt$, with the definitions introduced in (\ref{Constantes}). Moreover, from (\ref{ecuacion_casuistica}) we obtain that,
\begin{align*}
\dedtres^0(t)= - \frac{\lambda}{\theta + \rho} a^{\alpha \beta} \eab^0(t)- \frac{\lambda + 2\mu}{\theta + \rho} \edtres^0(t) - \frac{\theta}{\theta + \rho}  a^{\alpha\beta}\deab^0(t),
\end{align*}
 $\en \Omega, \ae$.

\item Let $p=0$ in (\ref{ecuacion_orden_fuerzas}). Grouping the terms multiplied by $\var^0$, taking into account (\ref{tensorA_tildes})--(\ref{tensorB_tildes}) and by step $(i)$ we find 
\begin{align}\nonumber
&\intO A^{ijkl}(0) \left(\ekl^0 \eij^0(\bv) + \ekl^1\eij^{-1}(\bv)  \right) \sqrt{a} dx + \intO \tilde{A}^{ijkl,1}\ekl^0\eij^{-1}(\bv) dx
\\ \nonumber
 &\qquad  +  \intO B^{ijkl}(0) \left(\dekl^0 \eij^0(\bv) + \dekl^1\eij^{-1}(\bv)  \right) \sqrt{a} dx + \intO \tilde{B}^{ijkl,1}\dekl^0\eij^{-1}(\bv)dx
 \\ \label{ec_step3}
 & \quad = \intO f^{i,0} v_i \sqrt{a} dx + \intG h^{i,1}v_i \sqrt{a} d\Gamma,
\end{align}
for all $\bv\in V(\Omega), \aes$. Taking  $\bv\in V(\Omega)$ such that it is independent of the transversal variable $x_3$, this is, such that we can identify $\bv$ with a function $\beeta\in V(\omega)$, we have by (\ref{eij_terminos_expansion}) that $\eij^{-1}(\bv)=0$. Moreover, since $\eatres^0=0$ by step $(ii)$, we have
\begin{align}\nonumber
&\intO A^{ijkl}(0)\ekl^0 \eij^0(\beeta) \sqrt{a} dx + \intO B^{ijkl}(0)\dekl^0\eij^0(\beeta) \sqrt{a} dx \\ \nonumber
& \quad = \intO \left( \lambda a^{\alpha \beta} a^{\sigma \tau} + \mu (a^{\alpha \sigma}a^{\beta \tau} + a^{\alpha \tau}a^{\beta\sigma})  \right) \est^0 \eab^0(\beeta) \sqrt{a} dx 
 + \intO \lambda a^{\alpha\beta}\edtres^0 \eab^0 (\beeta)\sqrt{a} dx
\\ \nonumber
& \qquad + \intO \left( \theta a^{\alpha \beta} a^{\sigma \tau} + \frac{\rho}{2} (a^{\alpha \sigma}a^{\beta \tau} + a^{\alpha \tau}a^{\beta\sigma})  \right) \dest^0 \eab^0(\beeta) \sqrt{a} dx
+ \intO \theta a^{\alpha\beta}\dedtres^0 \eab^0 (\beeta)\sqrt{a} dx 
\\ \label{ref2}
&\quad = \intO f^{i,0} \eta_i \sqrt{a} dx + \intG h^{i,1}\eta_i \sqrt{a} d\Gamma .
\end{align}
Using the expressions of $\edtres$ and its time derivative found in step $(ii)$, we have that
\begin{align*}\nonumber
 &\intO \left( \lambda a^{\alpha \beta} a^{\sigma \tau} + \mu (a^{\alpha \sigma}a^{\beta \tau} + a^{\alpha \tau}a^{\beta\sigma})  \right) \est^0 \eab^0(\beeta) \sqrt{a} dx 
 + \intO \lambda a^{\alpha\beta}\edtres^0 \eab^0 (\beeta)\sqrt{a} dx
\\ \nonumber
& \qquad + \intO \left( \theta a^{\alpha \beta} a^{\sigma \tau} + \frac{\rho}{2} (a^{\alpha \sigma}a^{\beta \tau} + a^{\alpha \tau}a^{\beta\sigma})  \right) \dest^0 \eab^0(\beeta) \sqrt{a} dx
+ \intO \theta a^{\alpha\beta}\dedtres^0 \eab^0 (\beeta)\sqrt{a} dx 
\\ 
&\quad =  \intO \left( \lambda a^{\alpha \beta} a^{\sigma \tau} + \mu (a^{\alpha \sigma}a^{\beta \tau} + a^{\alpha \tau}a^{\beta\sigma})  \right) \est^0 \eab^0(\beeta) \sqrt{a} dx 
 \\
 & \qquad + \intO \left( \theta a^{\alpha \beta} a^{\sigma \tau} + \frac{\rho}{2} (a^{\alpha \sigma}a^{\beta \tau} + a^{\alpha \tau}a^{\beta\sigma})  \right) \dest^0 \eab^0(\beeta) \sqrt{a} dx
\\ \nonumber
& \qquad + \intO \left(\lambda - \theta \frac{\lambda + 2\mu}{\theta + \rho } \right) \left( - \frac{\theta}{\theta + \rho} \left( a^{\sigma \tau }\est^0 + \Lambda\int_0^te^{-k(t-s)}a^{\sigma \tau}\est^0(s) ds \right)  \right) a^{\alpha\beta} \eab^0(\beeta) \sqrt{a} dx
\\
& \qquad - \intO \frac{\theta}{\theta + \rho} \left( \lambda a^{\sigma \tau}\est^0 + \theta a^{\sigma \tau}\dest^0  \right) a^{\alpha \beta} \eab^0 (\beeta) \sqrt{a} dx,
\end{align*}
which is equivalent to,
\begin{align*}
&\intO \left( \left(\lambda - \frac{\theta}{\theta + \rho} \left( \theta \Lambda + \lambda \right)\right)a^{\alpha\beta}a^{\sigma \tau} + \mu\ten  \right) \est^0\eab^0(\beeta)\sqrt{a}dx
\\
&\qquad + \intO \left( \frac{\theta \rho }{\theta + \rho} a^{\alpha \beta} a^{\sigma \tau}  + \frac{\rho}{2}\ten  \right) \dest^0\eab^0(\beeta)\sqrt{a}dx
\\
& \qquad - \intO \frac{\left( \theta \Lambda\right)^2}{\theta + \rho} \int_0^t e^{-k(t-s)}a^{\sigma \tau} \est^0(s) ds a^{\alpha \beta} \eab^0(\beeta) \sqrt{a} dx
\\ 
&\quad = \intO f^{i,0} \eta_i \sqrt{a} dx + \intG h^{i,1}\eta_i \sqrt{a} d\Gamma,
\end{align*}
hence, we obtain that
\begin{align*}
&\frac{1}{2}\intO \a \est^0\eab^0(\beeta)\sqrt{a}dx + \frac{1}{2}\intO \b \dest^0 \eab^0(\beeta)\sqrt{a}dx 
 \\
 & \qquad- \frac{1}{2}\int_0^te^{-k(t-s)}\intO \c \est^0(s)\eab^0(\beeta)\sqrt{a}dx ds 
 \\ & \quad = \intO f^{i,0}\eta_i \sqrt{a}dx + \intG h^{i,1} \eta_i \sqrt{a} d \Gamma, \ \forall\beeta\in V(\omega) \aes,
\end{align*}
where  $\a$ , $\b$ and $\c$ denote the contravariant components of the fourth order two-dimensional   tensors, defined in (\ref{tensor_a_bidimensional})--(\ref{tensor_c_bidimensional}).

Note that if $\beeta=(\eta_i)\in H^1(\omega)\times H^1(\omega) \times L^2(\omega),$ then
\begin{align*}
\gab(\beeta)\in L^2(\omega).
\end{align*}
Hence, the equalities in (\ref{et3})
\begin{align*}
\eab^0(t)=\gab(\bxi^0(t)) \ \textrm{and} \ \eab^0(\beeta(t))=\gab(\beeta(t)) \ \textrm{for all}\ \beeta\in V(\omega) \Forallt,
\end{align*}
 follow from the definitions (\ref{eij_terminos_expansion_u}), (\ref{eij_terminos_expansion}) and (\ref{def_gab}).

\item Assume that $V_0(\omega)=\{\bcero\}$. By the previous step we have the following variational problem: 

 Find $\bxi^0:[0,T] \times\omega \longrightarrow \mathbb{R}^3$ such that, 
\begin{align}\nonumber
&\bxi^0(t)\in V(\omega) \forallt,
\\ \nonumber
&\into \a \gst(\bxi^0)\gab(\beeta)\sqrt{a}dy + \into \b \gst(\dot{\bxi}^0) \gab(\beeta)\sqrt{a}dy
\\ \nonumber
 & \qquad- \int_0^te^{-k(t-s)}\into \c \gst(\bxi^0(s))\gab(\beeta)\sqrt{a}dy ds 
 \\ \label{ec_paso4}
 & \quad = \into p^{i,0}\eta_i \sqrt{a}dy, \ \forall\beeta\in V(\omega) \aes,
 \\ \nonumber
 & \bxi^0(0,\cdot)=\bxi_0^0(\cdot),
\end{align}
where $p^{i,0}$ is defined in (\ref{p0}).  This problem will be known as the two-dimensional variational problem for a viscoelastic membrane shell.

\item Assume that $V_0(\omega)\neq\{\bcero\}$. Taking $\beeta\in (V_0(\omega)\setminus\{\bcero\})$ in (\ref{ec_paso4}) we have that
\begin{align*}
\into p^{i,0}\eta_i \sqrt{a}dy = \intO f^{i,0}\eta_i \sqrt{a}dx + \intG h^{i,1} \eta_i \sqrt{a} d \Gamma=0.
\end{align*}
Hence, in order to avoid compatibility conditions between the applied forces we must take $f^{i,0}=0$ and $h^{i,1}=0$. Therefore, taking $\beeta=\bxi^0$ in the equation (\ref{ec_paso4}) leads to
\begin{align*} \nonumber
&\into \a \gst(\bxi^0)\gab(\bxi^0)\sqrt{a}dy + \into \b \gst(\dot{\bxi}^0) \gab(\bxi^0)\sqrt{a}dy
\\ 
& \qquad- \int_0^te^{-k(t-s)}\into \c \gst(\bxi^0(s))\gab(\bxi^0)\sqrt{a}dy ds =0.
\end{align*}
By  (\ref{condicion_inicial_def}) and the  first equality in (\ref{et3}), we have that $\gab(\bxi^0(0))=0.$ This initial condition  together with  the Theorem \ref{teorema_existencia_bidimensional} imply that $\gab(\bxi^0(t))=0 \Forallt$, that is, $\bxi^0\in V_0(\omega)$. Therefore, again by (\ref{et3}), we find that $\eab^0=\gab(\bxi^0)=0$. Moreover, by (\ref{eij_terminos_expansion_u}) and (\ref{edtres_cero}) we have that
\begin{equation*}
\d_3 u_3^1(t)=\edtres^0(t)=0 \en \Omega, \Forallt.
\end{equation*}
By the definition of $\eatres^0$ in (\ref{eij_terminos_expansion_u}) and steps $(i)$--$(ii)$ we have
\begin{align*}
\eatres^0=\frac{1}{2}\left( \d_\alpha \xi^0_3+ \d_3 u_\alpha^1 \right)+ b_\alpha^\sigma \xi^0_\sigma =0,
\end{align*}
hence,
\begin{align*}
\d_3 u_\alpha^1(t)=- \left(\d_\alpha \xi_3^0(t) + 2b_\alpha^\sigma \xi^0_\sigma(t)\right)  \en \Omega, \Forallt.
\end{align*}
Since we are assuming that $\bu^1(t)\in V(\Omega) \Forallt$ and since $\bxi^0$ is independent of $x_3$ by step $(i)$, there exists a field $\bxi^1(t)\in V(\omega) \Forallt$ such that
\begin{align*}
u_\alpha^1 (t)&= \xi^1_\alpha(t) - x_3 \left(\d_\alpha\xi_3^0(t) + 2 b_\alpha^\sigma \xi_\sigma^0(t) \right),
\\
u_3^1(t)&=\xi_3^1(t),
\end{align*}
in $\Omega, \Forallt$. Notice that this implies that $\xi^0_3(t)\in H^2(\Omega) \Forallt$. Now, since $\xi_\alpha^0=0$ on $\gamma_0$, then $\d_\nu\xi_3^0=0$, where $\d_\nu$ denotes the outer normal derivative along the boundary. Therefore, we have $\bxi^0(t)\in V_F(\omega) \Forallt$. Since $\eij^0=0$, coming back to the terms multiplied by $\var^0$ (see (\ref{ec_step3}) in step $(iii)$), we have 
\begin{align*}
\intO A^{ijkl}(0) \ekl^1\eij^{-1}(\bv)  \sqrt{a} dx +   \intO B^{ijkl}(0)  \dekl^1\eij^{-1}(\bv) \sqrt{a} dx =0,
\end{align*}
for all $\bv\in V(\Omega), \aes$. Notice that this equation is analogous to the one obtained in the step $(ii)$ involving the terms $\eij^1$ instead of the terms $\eij^0$ (see (\ref{ec_step2})). Therefore, using similar arguments, we conclude that 
\begin{align*}
\eatres^1(t)=0 \en \Omega, \Forallt,
\end{align*}
 and moreover,
 \begin{align*}
 \edtres^1(t)= - \frac{\theta}{\theta + \rho} \left( a^{\alpha \beta }\eab^1(t) + \Lambda\int_0^te^{-k(t-s)}a^{\alpha\beta}\eab^1(s) ds \right), \en \Omega, \Forallt,
 \end{align*}
  where $\Lambda$ and $k$ are defined  in (\ref{Constantes}). Furthermore,
 \begin{align*}
 \dedtres^1(t)= - \frac{\lambda}{\theta + \rho} a^{\alpha \beta} \eab^1(t)- \frac{\lambda + 2\mu}{\theta + \rho} \edtres^1(t) - \frac{\theta}{\theta + \rho} a^{\alpha\beta}\deab^1(t),
 \end{align*}
  $\en \Omega, \ae$.

Now by the the definitions in (\ref{eij_terminos_expansion_u}) in terms of $\xi^0_i$ and $\xi^1_i$ and replacing $\d_\beta b_\alpha^\sigma$ terms from (\ref{b_barra}), after some computations  we have that
\begin{align} \nonumber
\eab^1&= \frac{1}{2} \left( \d_\beta \xi_\alpha^1 + \d_\alpha \xi_\beta^1  \right) - \Gamma_{\alpha\beta}^\sigma \xi_\sigma^1 - b_{\alpha\beta} \xi_3^1 
- x_3\left( \d_{\alpha\beta}\xi_3^0 - \Gamma_{\alpha\beta}^\sigma\d_\sigma\xi_3^0 - b_\alpha^\sigma b_{\sigma\beta}\xi_3^0 \right.
\\ \label{ec_paso5} 
&\qquad \left.
+ b_\alpha^\sigma\left(\d_\beta\xi_\sigma^0 - \Gamma_{\beta\sigma}^\tau\xi_\tau^0    \right)   + b_\beta^\tau\left( \d_\alpha\xi_\tau^0 - \Gamma_{\alpha\tau}^\sigma \xi_\sigma^0 \right)+ b^\tau_{\beta|\alpha}\xi_\tau^0 \right).
\end{align}
Note that if $\beeta=(\eta_i)\in H^1(\omega)\times H^1(\omega) \times L^2(\omega),$ then (see (\ref{rab}))
\begin{align*}
\rho_{\alpha\beta} (\beeta) \in L^2(\Omega).
\end{align*}
Hence, by (\ref{def_gab}) for $\beeta=\bxi^1(t)$ and (\ref{rab}) for $\beeta=\bxi^0(t)$, it follows from (\ref{ec_paso5}) the equality 
\begin{equation*}
\eab^1(t)=\gab(\bxi^1(t))- x_3\rab(\bxi^0(t)) \en \Omega ,\Forallt.
\end{equation*}

\item Assume that $V_0(\Omega)\neq\{\bcero\}.$ Let $p=1$ in (\ref{ecuacion_orden_fuerzas}). Grouping the terms multiplied by $\var$, taking into account steps $(i)-(v)$ we have
\begin{align}\nonumber
&\intO A^{ijkl}(0)\left(\ekl^1\eij^0(\bv) + \ekl^2\eij^{-1}(\bv) \right) \sqrt{a} dx+ \intO \tilde{A}^{ijkl,1}\ekl^1\eij^{-1}(\bv)dx 
\\ \nonumber
&\qquad + \intO B^{ijkl}(0)\left(\dekl^1\eij^0(\bv) + \dekl^2\eij^{-1}(\bv) \right) \sqrt{a} dx+ \intO \tilde{B}^{ijkl,1}\dekl^1\eij^{-1}(\bv)dx 
\\ \label{ref3}
&\quad =\intO f^{i,1}v_i \sqrt{a}dx + \intG h^{i,2} v_i \sqrt{a}d\Gamma,
\end{align}
for all $\bv\in V(\Omega), \aes$. Taking $\bv=\beeta\in V(\omega)$, this is, $\bv$ independent of $x_3$, by (\ref{eij_terminos_expansion}) we obtain
\begin{align*}
&\intO A^{ijkl}(0)\ekl^1\eij^0(\beeta) \sqrt{a} dx
+ \intO B^{ijkl}(0)\dekl^1\eij^0(\beeta)  \sqrt{a} dx
\\
&\quad =\intO f^{i,1}\eta_i \sqrt{a}dx + \intG h^{i,2} \eta_i \sqrt{a}d\Gamma,
\end{align*}
for all $\beeta\in V(\omega), \aes$. Since $\eatres^1=0$ by $(v)$ we obtain
\begin{align*}
&\intO A^{ijkl}(0)\ekl^1\eij^0(\beeta) \sqrt{a} dx
+ \intO B^{ijkl}(0)\dekl^1\eij^0(\beeta)  \sqrt{a} dx
\\
&\quad =\intO \left( \lambda a^{\alpha \beta} a^{\sigma \tau} + \mu (a^{\alpha \sigma}a^{\beta \tau} + a^{\alpha \tau}a^{\beta\sigma})  \right) \est^1 \eab^0(\beeta) \sqrt{a} dx 
 + \intO \lambda a^{\alpha\beta}\edtres^1 \eab^0 (\beeta)\sqrt{a} dx
\\ \nonumber
& \qquad + \intO \left( \theta a^{\alpha \beta} a^{\sigma \tau} + \frac{\rho}{2} (a^{\alpha \sigma}a^{\beta \tau} + a^{\alpha \tau}a^{\beta\sigma})  \right) \dest^1 \eab^0(\beeta) \sqrt{a} dx
+ \intO \theta a^{\alpha\beta}\dedtres^1 \eab^0 (\beeta)\sqrt{a} dx 
\\
&\quad = \intO f^{i,1} \eta_i \sqrt{a} dx + \intG h^{i,2}\eta_i \sqrt{a} d\Gamma ,
\end{align*}
for all $\beeta\in V(\omega), \aes$, which is analogous to the expression obtained in (\ref{ref2}). Therefore, following the same arguments made there, taking into account $(v)$,  we find that
\begin{align}\nonumber
&\into \a \gst(\bxi^1)\gab(\beeta)\sqrt{a}dy + \into \b \gst(\dot{\bxi}^1) \gab(\beeta)\sqrt{a}dy
\\ \nonumber
& \qquad- \int_0^te^{-k(t-s)}\into \c \gst(\bxi^1(s))\gab(\beeta)\sqrt{a}dy ds 
 \\ \label{et6}
&\quad = \intO f^{i,1} \eta_i \sqrt{a} dx + \intG h^{i,2}\eta_i \sqrt{a} d\Gamma ,
\end{align}
for all $\beeta\in V(\omega), \aes$,  where the contravariant components of the fourth order two-dimensional tensors $\a,\b,\c$ are defined  in  (\ref{tensor_a_bidimensional})--(\ref{tensor_c_bidimensional}). Taking $\beeta\in \left( V_0(\omega)\setminus\{\bcero\} \right)$ we have that
\begin{align*}
\intO f^{i,1} \eta_i \sqrt{a} dx + \intG h^{i,2}\eta_i \sqrt{a} d\Gamma =0,
\end{align*}
hence, in order to avoid compatibility conditions between the applied forces we must take $f^{i,1}=0$ and $h^{i,2}=0$. Therefore, letting $\beeta=\bxi^1$ in (\ref{et6}) leads to
\begin{align*}
&\into \a \gst(\bxi^1)\gab(\bxi^1)\sqrt{a}dy + \into \b \gst(\dot{\bxi}^1) \gab(\bxi^1)\sqrt{a}dy
\\
& \qquad- \int_0^te^{-k(t-s)}\into \c \gst(\bxi^1(s))\gab(\bxi^1(t))\sqrt{a}dy ds =0.
\end{align*}
By (\ref{condicion_inicial_def}) and the relation (\ref{ref1}) found in the step $(v)$, we obtain that $ \gab(\bxi^1(0))=0$, hence, by the Theorem \ref{teorema_existencia_bidimensional} we deduce that $ \gab(\bxi^1(t))=0 \Forallt$. Therefore,
\begin{align*}
\bxi^1(t)\in V_0(\omega) \Forallt.
\end{align*}

\item On one hand, coming back to the equation (\ref{ref3}), with $f^{i,1}=0$ and $h^{i,2}=0$, leads to
\begin{align}\nonumber
&\intO A^{ijkl}(0)\left(\ekl^1\eij^0(\bv) + \ekl^2\eij^{-1}(\bv) \right) \sqrt{a} dx+ \intO \tilde{A}^{ijkl,1}\ekl^1\eij^{-1}(\bv)dx 
\\ \nonumber
&\qquad + \intO B^{ijkl}(0)\left(\dekl^1\eij^0(\bv) + \dekl^2\eij^{-1}(\bv) \right) \sqrt{a} dx+ \intO \tilde{B}^{ijkl,1}\dekl^1\eij^{-1}(\bv)dx =0
\end{align}
Given $\beeta\in V_F(\omega)$, we define $\bv(\beeta)=(v_i(\beeta))$ as
\begin{align*}
v_\alpha(\beeta)&:= x_3 \left( 2b_\alpha^\sigma \eta_\sigma + \d_\alpha\eta_3 \right),\\
v_3(\beeta)&:=0,
\end{align*}
and take $\bv=\bv(\beeta)$ in the previous equation, leading to (see (\ref{eij_terminos_expansion}))
\begin{align}\nonumber
&\intO A^{ijkl}(0)\ekl^1\eij^0(\bv(\beeta)) \sqrt{a}dx + 4\intO  A^{\alpha 3\sigma 3}(0)\estres^2\left( b_\alpha^\tau\eta_\tau + \frac{1}{2} \d_\alpha\eta_3   \right) \sqrt{a} dx 
\\\nonumber
& \qquad + 4\intO \tilde{A}^{\alpha 3 \sigma 3,1}\estres^1\left( b_\alpha^\tau\eta_\tau + \frac{1}{2} \d_\alpha\eta_3   \right)dx 
\\ \nonumber
&\qquad +\intO B^{ijkl}(0)\dekl^1\eij^0(\bv(\beeta)) \sqrt{a}dx + 4\intO  B^{\alpha 3\sigma 3}(0)\destres^2\left( b_\alpha^\tau\eta_\tau + \frac{1}{2} \d_\alpha\eta_3   \right) \sqrt{a} dx 
\\ \label{ref4}
& \qquad + 4\intO \tilde{B}^{\alpha 3 \sigma 3,1}\destres^1\left( b_\alpha^\tau\eta_\tau + \frac{1}{2} \d_\alpha\eta_3   \right)dx =0,
\end{align}
for all $\beeta\in V_F(\omega), \aes$. On the other hand, let $p=2$ in  (\ref{ecuacion_orden_fuerzas}). Grouping the terms multiplied by $\var^2$ and using steps $(i)$ and $(v)$ we find that
\begin{align*}
&\intO A^{ijkl}(0) \left(\ekl^1\eij^1(\bv) + \ekl^2\eij^0(\bv) + \ekl^{3}\eij^{-1}(\bv)  \right)\sqrt{a}dx 
\\
&\qquad + \intO \tilde{A}^{ijkl,1}\left(\ekl^1\eij^0(\bv) + \ekl^2\eij^{-1}(\bv)   \right) dx +
\intO \tilde{A}^{ijkl,2} \ekl^1\eij^{-1}(\bv) dx
\\
& \qquad
+\intO  B^{ijkl}(0) \left(\dekl^1\eij^1(\bv) + \dekl^2\eij^0(\bv) + \dekl^{3}\eij^{-1}(\bv)  \right)\sqrt{a}dx 
\\
&\qquad + \intO\tilde{B}^{ijkl,1}\left(\dekl^1\eij^0(\bv) + \dekl^2\eij^{-1}(\bv)   \right) dx +
\intO \tilde{B}^{ijkl,2} \dekl^1\eij^{-1}(\bv) dx
\\
& \quad  = \intO f^{i,2} v_i \sqrt{a} dx + \intG h^{i,3}v_i \sqrt{a} d\Gamma ,
\end{align*}
for all $\bv\in V(\Omega), \aes$. Consider now any $\bv$ which can be identified with a function $\beeta\in V_F(\omega)$; hence by steps $(i)$, $(v)$ and (\ref{eij_terminos_expansion}) we have
\begin{align*}
&\intO A^{ijkl}(0) \ekl^1\eij^1(\beeta)\sqrt{a}dx + 4\intO A^{\alpha 3 \sigma 3}(0) \estres^2 \left(  b_\alpha^\tau\eta_\tau + \frac{1}{2} \d_\alpha\eta_3   \right) \sqrt{a}dx
\\
&\qquad + \intO \tilde{A}^{ijkl,1}\estres^1 \left(  b_\alpha^\tau\eta_\tau + \frac{1}{2} \d_\alpha\eta_3   \right) dx
\\
& \qquad
+\intO B^{ijkl}(0) \dekl^1\eij^1(\beeta)\sqrt{a}dx + 4\intO B^{\alpha 3 \sigma 3}(0) \destres^2 \left(  b_\alpha^\tau\eta_\tau + \frac{1}{2} \d_\alpha\eta_3   \right) \sqrt{a}dx
\\
&\qquad + \intO \tilde{B}^{ijkl,1}\destres^1 \left(  b_\alpha^\tau\eta_\tau + \frac{1}{2} \d_\alpha\eta_3   \right) dx
\\
& \quad  = \into p^{i,2} \eta_i \sqrt{a} dy , 
\end{align*}
 for all $\beeta\in V_F( \omega), \aes$, where $p^{i,2}$ is defined  in (\ref{p2}). By subtracting (\ref{ref4}), we obtain
\begin{align} \nonumber
&\intO A^{ijkl}(0) \ekl^1 \left(\eij^1(\beeta) - \eij^0(\bv(\beeta)) \right) \sqrt{a}dx 
 + \intO B^{ijkl}(0) \dekl^1 \left(\eij^1(\beeta) - \eij^0(\bv(\beeta)) \right) \sqrt{a}dx 
 \\ \label{ref5}
 & \quad  = \into p^{i,2} \eta_i \sqrt{a} dy ,
\end{align}
 for all $\beeta\in V_F( \omega), \aes.$ Now,  by step $(v)$ and (\ref{eij_terminos_expansion})  we have that
\begin{align*}
A^{ijkl}(0)\ekl^1 \left( \eij^1(\beeta) - \eij^0(\bv(\beeta))  \right)
&= A^{\alpha\beta\sigma\tau}(0)\est^1\left(\eab^1 (\beeta)- \eab^0(\bv(\beeta) \right)  \\
&\quad+ A^{\alpha\beta 33}(0)\edtres^1\left(\eab^1 (\beeta)- \eab^0(\bv(\beeta)) \right).
\end{align*}
We also have the analogous equality for the components of the viscosity tensor multiplying the time derivatives of the strain components. Moreover, by steps $(v)$ and $(vi)$ we have
\begin{align} \label{est_1_rst}
\est^1(t)= -x_3 \rst(\bxi^0(t)) \Forallt. 
\end{align}
Furthermore, by (\ref{eij_terminos_expansion}) we also find that
\begin{align*}
\eab^1(\beeta)-\eab^0(\bv(\beeta))&= x_3 \left( b^\sigma_{\beta|\alpha} \eta_\sigma + b_\alpha^\sigma b_{\sigma\beta}\eta_3  \right) \\
& \quad- x_3 \left(\d_\alpha(b_\beta^\tau \eta_\tau) + \d_\beta(b_\alpha^\sigma\eta_\sigma) + \d_{\alpha\beta}\eta_3 - \Gamma_{\alpha\beta}^\sigma \d_\sigma \eta_3 -2\Gamma_{\alpha\beta}^\sigma b_\sigma^\tau \eta_\tau   \right),
\end{align*}
and making some calculations we conclude that 
\begin{align*}
\eab^1(\beeta) - \eab^0(\bv(\beeta)) = - x_3 \rab(\beeta), \ \forall \beeta\in V_F(\omega).
\end{align*}
Therefore, the left-hand side of the equation (\ref{ref5}) leads to
\begin{align} \nonumber
&\intO A^{ijkl}(0) \ekl^1 \left(\eij^1(\beeta) - \eij^0(\bv(\beeta)) \right) \sqrt{a}dx 
 + \intO B^{ijkl}(0) \dekl^1 \left(\eij^1(\beeta) - \eij^0(\bv(\beeta)) \right) \sqrt{a}dx 
 \\ \nonumber
 & \quad  =\intO \left( \lambda a^{\alpha \beta} a^{\sigma \tau} + \mu (a^{\alpha \sigma}a^{\beta \tau} + a^{\alpha \tau}a^{\beta\sigma})  \right)\est^1\left( -x_3 \rab(\beeta)\right) \sqrt{a} dx 
 \\ \nonumber
 &\qquad + \intO \lambda a^{\alpha\beta}\edtres^1 \left(-x_3 \rab(\beeta)  \right)\sqrt{a} dx
 \\ \nonumber
 & \qquad + \intO \left( \theta a^{\alpha \beta} a^{\sigma \tau} + \frac{\rho}{2} (a^{\alpha \sigma}a^{\beta \tau} + a^{\alpha \tau}a^{\beta\sigma})  \right) \dest^1 \left( -x_3 \rab(\beeta)\right) \sqrt{a} dx
 \\ \label{ref6}
 &\qquad + \intO \theta a^{\alpha\beta}\dedtres^1 \left(-x_3 \rab(\beeta)  \right)\sqrt{a} dx .
\end{align}
Now, by the findings in step $(v)$, we have that (\ref{ref6}) leads to
\begin{align*}
&\intO \left( \left(\lambda - \frac{\theta}{\theta + \rho} \left( \theta \Lambda + \lambda \right)\right)a^{\alpha\beta}a^{\sigma \tau} + \mu\ten  \right) \est^1\left(-x_3 \rab(\beeta)  \right)\sqrt{a}dx
\\
&\qquad + \intO \left( \frac{\theta \rho }{\theta + \rho} a^{\alpha \beta} a^{\sigma \tau}  + \frac{\rho}{2}\ten  \right) \dest^1\left(-x_3 \rab(\beeta)  \right)\sqrt{a}dx
\\
& \qquad - \intO \frac{\left( \theta \Lambda\right)^2}{\theta + \rho} \int_0^t e^{-k(t-s)}a^{\sigma \tau} \est^1(s) ds \left(-x_3 a^{\alpha \beta} \rab(\beeta)  \right) \sqrt{a} dx,
\end{align*}
which using (\ref{est_1_rst}) is equivalent to
\begin{align*}
&\intO \frac{x_3^2}{2}\a \rst(\bxi^0)\rab(\beeta)\sqrt{a}dx + \intO \frac{x_3^2}{2}\b \rst(\dot{\bxi}^0)\rab(\beeta)\sqrt{a}dx 
 \\
 & \qquad- \int_0^te^{-k(t-s)}\intO \frac{x_3^2}{2}\c \rst(\bxi^0(s))\rab(\beeta)\sqrt{a}dx ds 
 \\ & \quad = \frac{1}{3}\into \a \rst(\bxi^0)\rab(\beeta)\sqrt{a}dy + \frac{1}{3}\into \b \rst(\dot{\bxi}^0)\rab(\beeta)\sqrt{a}dy 
  \\
  & \qquad- \frac{1}{3}\int_0^te^{-k(t-s)}\into \c \rst(\bxi^0(s))\rab(\beeta)\sqrt{a}dy ds, 
\end{align*}
 for all $\beeta\in V_F(\omega), \aes,$ where  $\a$ , $\b$ and $\c$ denote the contravariant components of the fourth order two-dimensional  tensors, defined in (\ref{tensor_a_bidimensional})--(\ref{tensor_c_bidimensional}). Hence, we have obtained the following variational problem:
 
  Find $\bxi^0:[0,T] \times\omega \longrightarrow \mathbb{R}^3$ such that
\begin{align}\nonumber
&\bxi^0(t)\in V_F(\omega) \forallt,
\\ \nonumber
&\frac{1}{3}\into \a \rst(\bxi^0)\rab(\beeta)\sqrt{a}dy + \frac{1}{3}\into \b \rst(\dot{\bxi}^0)\rab(\beeta)\sqrt{a}dy 
  \\ \nonumber
  & \qquad- \frac{1}{3}\int_0^te^{-k(t-s)}\into \c \rst(\bxi^0(s))\rab(\beeta)\sqrt{a}dy ds 
  \\ \label{et1}
& \quad= \into p^{i,2}\eta_i \sqrt{a}dy \ \forall \beeta\in V_F(\omega), \aes, 
\\ \nonumber
& \bxi^0(0,\cdot)=\bxi^0_0(\cdot).
\end{align}
 This problem will be known as the two-dimensional variational problem for a viscoelastic flexural shell.
 
\end{enumerate}

\end{proof}

\begin{remark}
The mathematical variational models found in (\ref{ec_paso4}) and in (\ref{et1}) show a long-term memory that takes into account the deformations in  previous times, represented by an integral on the time variable. Notice that the weight coefficient term makes the older strain states less influential than the newer ones. Analogous behavior has been presented in beam models for the bending-stretching of viscoelastic rods \cite{AV}, obtained by using asymptotic methods as well. Also, this kind of viscoelasticity has been described in \cite{DL,Pipkin}, for example.
\end{remark}

\section{Existence and uniqueness of the solution of the two-dimensional problems} \setcounter{equation}{0}\label{Existencia}

In what follows, we study the existence and uniqueness of solution of the two-dimensional limit problems found in the previous section: the membrane and flexural shell cases. To that aim, we first give the following result regarding the ellipticity of the fourth order two-dimensional tensors defined by their contravariant components in  (\ref{tensor_a_bidimensional})--(\ref{tensor_c_bidimensional}).

 \begin{theorem} 
 Let $\omega$  be a domain in $\mathbb{R}^2$, let $\btheta\in\mathcal{C}^1(\bar{\omega};\mathbb{R}^3)$ be an injective mapping such that the two vectors $\ba_\alpha=\d_\alpha\btheta$ are linearly independent at all points of $\bar{\omega}$, let $a^{\alpha\beta}$ denote the contravariant components of the metric tensor of $S=\btheta(\bar{\omega})$.  Let us consider the contravariant components of the scaled fourth order two-dimensional  tensors of the shell, $\a,\b,$
  defined in (\ref{tensor_a_bidimensional})--(\ref{tensor_b_bidimensional}). Assume that $\lambda\geq0$ and $\mu,\theta,\rho>0$. Then there exist two constants $c_e>0$ and $c_v>0$   independent of the variables and $\var$,  such that
 \begin{align} \label{tensor_a_elip}
   \sum_{\alpha,\beta}|t_{\alpha\beta}|^2\leq c_e \a(\by)t_{\sigma\tau}t_{\alpha\beta},
   \\\label{tensor_b_elip}
   \sum_{\alpha,\beta}|t_{\alpha\beta}|^2 \leq c_v \b(\by)t_{\sigma\tau}t_{\alpha\beta},
 \end{align}
 for all $\by\in\bar{\omega}$ and all $\bt=(t_{\alpha\beta})\in\mathbb{S}^2$.
 \end{theorem}

\begin{remark}The proof of this result is straightforward following similar arguments as in  Theorem 3.3-2, \cite{Ciarlet4b}.

\end{remark}

We shall present the limit problems in a de-scaled form. The details of the convergence and the physical interpretation of the solutions for those problems are subject of forthcoming papers (\cite{eliptico,flexural,generalizada}). There we shall see that in fact, the subspace which plays the key role in differentiating viscoelastic membrane shells from viscoelastic flexural shells is $V_F(\omega)$ instead of $V_0(\omega)$, as happened in the elastic case (see \cite{Ciarlet4b}).
\subsection{Viscoelastic membrane shell}
Let us first consider that $V_F(\omega)=\{\bcero\}$.
In order to obtain a well posed problem we must consider a larger space, completion of $V(\omega)$ , which will be denoted by $V_M(\omega)$. Specifically, we will distinguish the different types of membranes  depending on the type of middle surface of the family of shells and the subset where the boundary condition of place is considered. For example, if the middle surface $S$ is elliptic and $\gamma=\gamma_0$, we take $V_M(\omega):=H^1_0(\omega)\times H^1_0(\omega)\times L^2(\omega)$. In this type of membranes it is verified the two-dimensional Korn's type inequality (see, for example, Theorem 2.7-3, \cite{Ciarlet4b}): there exists a constant $c_M=c_M(\omega,\btheta)$ such that
 \begin{align} \label{Korn_elipticas}
 \left( \sum_\alpha||\eta_\alpha||^2_{1,\omega} + ||\eta_3||_{0,\omega}^2       \right)^{1/2} \leq c_M \left( \sum_{\alpha,\beta} ||\gab(\beeta)||_{0,\omega}^2  \right)^{1/2} \ \forall \beeta\in V_M(\omega). 
 \end{align}
 Complete studies will be presented in detail in two forthcoming papers (\cite{eliptico,generalizada}). We can enunciate the de-scaled variational problem for a viscoelastic membrane shell:

\begin{problem}\label{problema_ab_eps}
 Find $\bxi^\var:[0,T] \times\omega \longrightarrow \mathbb{R}^3$ such that,
    \begin{align*}\nonumber
    & \bxi^\var(t,\cdot)\in V_M(\omega) \forallt,\\ \nonumber
   &\var\int_{\omega} \aeps\gst(\bxi^\var)\gab(\beeta)\sqrt{a}dy +\var\int_{\omega}\beps\gst(\dot{\bxi}^\var)\gab(\beeta)\sqrt{a}dy
   \\\nonumber
    & \qquad- \var\int_0^te^{-k(t-s)}\into \ceps \gst(\bxi^\var(s))\gab(\beeta)\sqrt{a}dy ds 
    \\ 
   &\quad=\int_{\omega}p^{i,\var}\eta_i\sqrt{a}dy \ \forall \beeta=(\eta_i)\in V_M(\omega), \aes,
    \\
    &\bxi^\var(0,\cdot)=\bxi^\var_0(\cdot),
   \end{align*}

   where,
   \begin{align*}
   &\gab(\beeta):= \frac{1}{2}(\d_\alpha\eta_\beta + \d_\beta\eta_\alpha) - \Gamma_{\alpha\beta}^\sigma\eta_\sigma -b_{\alpha\beta}\eta_3,
   \\
   & p^{i,\var}(t):=\int_{-\var}^{\var}\fb^{i,\var}(t)dx_3^\var +h_+^{i,\var}(t)+h_-^{i,\var}(t) \ \textrm{and} \ h_{\pm}^{i,\var}(t)=\bh^{i,\var}(t,\cdot,\pm \var),
   \end{align*}
  and where the contravariant components of the fourth order two-dimensional tensors $\aeps,$ $\beps, $ $\ceps$ are defined as  rescaled versions of  (\ref{tensor_a_bidimensional})--(\ref{tensor_c_bidimensional}). The space $V_M(\omega)$ denotes a space completion of $V(\omega)$ where the viscoelastic membrane problem is well posed (to be detailed in forthcoming papers).
\end{problem}

\begin{theorem} \label{Th_exist_unic_bid_cero}
Let $\omega$  be a domain in $\mathbb{R}^2$, let $\btheta\in\mathcal{C}^2(\bar{\omega};\mathbb{R}^3)$ be an injective mapping such that the two vectors $\ba_\alpha=\d_\alpha\btheta$ are linearly independent at all points of $\bar{\omega}$. Let $\fb^{i,\var}\in L^{2}(0,T; L^2(\Omega^\var)) $, $\bh^{i,\var}\in L^{2}(0,T; L^2(\Gamma_1^\var))$, where $\Gamma_1^\var:= \Gamma_+^\var\cup\Gamma_-^\var$. Let  $\bxi_0^\var\in V_M(\omega). $  Then the Problem \ref{problema_ab_eps}, has a unique solution  $\bxi^\var\in W^{1,2}(0,T;V_M(\omega))$.  In addition to that, if $\dot{\fb}^{i,\var}\in L^{2}(0,T; L^2(\Omega^\var)) $, $\dot{\bh}^{i,\var}\in L^{2}(0,T; L^2(\Gamma_1^\var))$, then $\bxi^\var\in W^{2,2}(0,T;V_M(\omega))$. 
\end{theorem}
\begin{proof}
 Let us consider the bilinear forms  $a^\var,b^\var,c^\var: V_M(\omega)\times V_M(\omega)\longrightarrow \mathbb{R}$ defined by,
\begin{align*}
a^\var(\bxi^\var,\beeta)&: = \var\int_{\omega} \aeps\gst(\bxi^\var)\gab(\beeta)\sqrt{a}dy,
\\
b^\var(\bxi^\var,\beeta )&:=\var\int_{\omega}\beps\gst({\bxi}^\var)\gab(\beeta)\sqrt{a}dy,
\\
c^\var(\bxi^\var,\beeta)&:=\var\into \ceps \gst(\bxi^\var)\gab(\beeta)\sqrt{a}dy  ,
\end{align*}
 for all $\bxi^\var,\beeta\in V_M(\omega) $ and for each $\var>0$. Therefore the Problem \ref{problema_ab_eps} can be cast into an analogous framework of the formulation (\ref{ecuacion_operadores})--(\ref{condicion_operadores}), since $p^{i,\var}\in L^2(0,T;L^2(\omega))$ and by the ellipticity of the two-dimensional tensors in (\ref{tensor_a_elip})--(\ref{tensor_b_elip}). Therefore, combining a Korn's type inequality (see (\ref{Korn_elipticas}) for the elliptic case) with similar arguments as in the proof of the Theorem \ref{teorema_existencia_bidimensional}, we  find  that the Problem \ref{problema_ab_eps} has uniqueness of solution and such that $\bxi^\var\in W^{1,2}(0,T;V_M(\omega))$. Moreover, with the additional regularity of $f^{i,\var}$ and $h^{i,\var}$, we conclude that $\bxi^\var\in W^{2,2}(0,T;V_M(\omega))$.
\end{proof}

\subsection{Viscoelastic flexural shell}

Let us consider now that the space $ V_F(\omega) $ contains non-zero functions. Therefore, we can enunciate the de-scaled variational problem for a viscoelastic flexural shell:

\begin{problem}\label{problema_flexural_eps}
 Find $\bxi^\var:[0,T] \times\omega \longrightarrow \mathbb{R}^3$ such that,
    \begin{align*}\nonumber
    & \bxi^\var(t,\cdot)\in V_F(\omega) \forallt,\\ \nonumber
   &\frac{\var^3}{3}\int_{\omega} \aeps\rst(\bxi^\var)\rab(\beeta)\sqrt{a}dy +\frac{\var^3}{3}\int_{\omega}\beps\rst(\dot{\bxi}^\var)\rab(\beeta)\sqrt{a}dy
    \\ \nonumber
      &- \frac{\var^3}{3}\int_0^te^{-k(t-s)}\into \c \rst(\bxi^\var(s))\rab(\beeta)\sqrt{a}dyds
   \\
   &\quad=\int_{\omega}p^{i,\var}\eta_i\sqrt{a}dy \ \forall \beeta=(\eta_i)\in V_F(\omega),\aes,
    \\
    &\bxi^\var(0,\cdot)=\bxi^\var_0(\cdot),
   \end{align*} 

   where,
   \begin{align*}
   &\rho_{\alpha\beta}(\beeta):= \d_{\alpha\beta}\eta_3 - \Gamma_{\alpha\beta}^\sigma \d_\sigma\eta_3 - b_\alpha^\sigma b_{\sigma\beta} \eta_3 + b_\alpha^\sigma (\d_\beta\eta_\sigma- \Gamma_{\beta\sigma}^\tau \eta_\tau) + b_\beta^\tau(\d_\alpha\eta_\tau-\Gamma_{\alpha\tau}^\sigma\eta_\sigma ) + b^\tau_{\beta|\alpha} \eta_\tau,
   \\
   & p^{i,\var}(t):=\int_{-\var}^{\var}\fb^{i,\var}(t)dx_3^\var+h_+^{i,\var}(t)+h_-^{i,\var}(t) \ \textrm{and} \ h_{\pm}^{i,\var}(t)=\bh^{i,\var}(t,\cdot,\pm \var),
   \end{align*}
 and where the contravariant components of the fourth order two-dimensional tensors $\aeps,$ $\beps, $ $\ceps$ are defined as rescaled versions of  (\ref{tensor_a_bidimensional})--(\ref{tensor_c_bidimensional}).  
\end{problem}

If $\btheta\in\mathcal{C}^3(\bar{\omega};\mathbb{R}^3),$ it is verified the following Korn's type inequality (see, for example, Theorem 2.6-4, \cite{Ciarlet4b}): there exists a constant $c=c(\omega, \gamma_0, \btheta)$ such that
 \begin{align} \label{Korn_flexural}
 \left( \sum_\alpha||\eta_\alpha||^2_{1,\omega} + ||\eta_3||_{2,\omega}^2       \right)^{1/2} \leq c \left( \sum_{\alpha,\beta} ||\rab(\beeta)||_{0,\omega}^2  \right)^{1/2} \ \forall \beeta\in V_F(\omega). 
 \end{align}
\begin{theorem} \label{Th_exist_unic_bid_dos}
Let $\omega$  be a domain in $\mathbb{R}^2$, let $\btheta\in\mathcal{C}^3(\bar{\omega};\mathbb{R}^3)$ be an injective mapping such that the two vectors $\ba_\alpha=\d_\alpha\btheta$ are linearly independent at all points of $\bar{\omega}$. Let $\fb^{i,\var}\in L^{2}(0,T; L^2(\Omega^\var)) $, $\bh^{i,\var}\in L^{2}(0,T; L^2(\Gamma_1^\var))$, where $\Gamma_1^\var:= \Gamma_+^\var\cup\Gamma_-^\var$. Let  $\bxi_0^\var\in V_F(\omega). $  Then the Problem \ref{problema_flexural_eps}, has a unique solution  $\bxi^\var\in W^{1,2}(0,T;V_F(\omega))$.  In addition to that, if $\dot{\fb}^{i,\var}\in L^{2}(0,T; L^2(\Omega^\var)) $, $\dot{\bh}^{i,\var}\in L^{2}(0,T; L^2(\Gamma_1^\var))$, then $\bxi^\var\in W^{2,2}(0,T;V_F(\omega))$. 
\end{theorem}
\begin{proof}
 Let us consider the bilinear forms $a^\var,b^\var,c^\var: V_F(\omega)\times V_F(\omega)\longrightarrow \mathbb{R}$  defined by,
\begin{align*}
a^\var(\bxi^\var,\beeta)&: = \frac{\var^3}{3}\int_{\omega} \aeps\rst(\bxi^\var)\rab(\beeta)\sqrt{a}dy,
\\
b^\var(\bxi^\var,\beeta )&:=\frac{\var^3}{3}\int_{\omega}\beps\rst({\bxi}^\var)\rab(\beeta)\sqrt{a}dy,
\\
c^\var(\bxi^\var,\beeta)&:=\frac{\var^3}{3}\into \ceps \rst(\bxi^\var)\rab(\beeta)\sqrt{a}dy ,
\end{align*}
for all $\bxi^\var,\beeta\in V_F(\omega) $ and for each $\var>0$. Therefore the Problem \ref{problema_flexural_eps} can be cast into an analogous framework of the formulation (\ref{ecuacion_operadores})--(\ref{condicion_operadores}), since $p^{i,\var}\in L^2(0,T;L^2(\omega))$  and by the ellipticity of the two-dimensional tensors in (\ref{tensor_a_elip})--(\ref{tensor_b_elip}). Therefore, combining a Korn's type inequality (see (\ref{Korn_flexural})) with similar arguments as in the proof of the Theorem \ref{teorema_existencia_bidimensional}, we  find that the Problem \ref{problema_flexural_eps} has uniqueness of solution and such that $\bxi^\var\in W^{1,2}(0,T;V_F(\omega))$. Moreover, with the additional regularity of $f^{i,\var}$ and $\bh^{i,\var}$, we conclude that $\bxi^\var\in W^{2,2}(0,T;V_F(\omega))$.
\end{proof}

\section{Conclusions} \setcounter{equation}{0} \label{conclusiones}

We have found limit two-dimensional  models for  viscoelastic membrane shells  and viscoelastic flexural shells. To this end we used the  asymptotic expansion method to identify the variational equations from the scaled three-dimensional viscoelastic shell problem. We have provided an analysis of the existence and uniqueness of solution for the three-dimensional problems and announced the corresponding results for the two-dimensional limit problems as well. Particularly interesting is that in the process of passing to the limit a long-term memory arises naturally (see (\ref{ec_paso4}) and  (\ref{et1})). Long-term memory is a well known phenomenon associated to a variety of viscoelastic materials  that takes into account the deformations of  previous times, represented by an integral on the time variable. Analogous behavior has been presented in beam models for the bending-stretching of viscoelastic rods \cite{AV}, obtained by using asymptotic methods as well.  Also, this kind of viscoelasticity has been described in \cite{DL,Pipkin}, for example.

 As the viscoelastic case differs from the elastic case on time dependent constitutive law and external forces, we must consider the possibility that these models generalize the elastic case (studied in \cite{Ciarlet4b}). However, as the reader can easily check, when the ordinary differential equation (\ref{ecuacion_casuistica}) was presented, we  had to use assumptions that make it impossible to consider the elastic case. For instance, we could try to reduce the viscoelastic model to the elastic case by neglecting the viscosity constants and considering the various functions involved to be stationary. We show in the Remark \ref{nota_desvio}, the last step where these arguments can be considered that, indeed, we would obtain the same models  obtained in \cite{Ciarlet4b} for the corresponding elastic cases. Nevertheless, in what follows, the viscosity coefficient $\theta$  can  not be zero, so  the same proof can not be followed from that point. Hence, the viscoelastic and elastic problems must be treated separately in order to reach reasonable and justified conclusions.

The asymptotic approaches need to be mathematically justified in order to ensure robust results. To this end,  guided by the formal analysis developed in this paper, a more deep and robust study including convergence theorems will be presented in forthcoming papers (\cite{eliptico,flexural,generalizada}), regarding the different cases that have appeared in this work.

 The formal asymptotic procedure made in this work has placed the two dimensional limit equations for the membrane case on spaces where the problems were not well posed, so we need to find completions for these spaces. This will be done by  taking into account the type of the middle surface of the family of shells and the subset where the boundary condition of place is considered. Therefore, on one hand, we shall study in \cite{eliptico} the case when $S$ is elliptic and when $\gamma_0=\gamma$, this is $V_0(\omega)= \{\bcero\}$ (which implies $V_F(\omega)=\{\bcero\}$). These are known as viscoelastic elliptic membrane shells. On the other hand, in \cite{generalizada} we shall consider the cases when the membrane is not elliptic or $\gamma_0\neq\gamma$ but still $V_F(\omega)= \{\bcero\}$. For these cases, additional spaces must be considered in order to obtain well posed problems. They are the so-called  viscoelastic generalized membranes, where we also distinguish the cases where $V_0(\omega)$ contains only the zero function (first kind) or not (second kind). Further, regarding the case where the space $V_F(\omega)$ contains non-zero functions, in \cite{flexural} we shall study the problem of  viscoelastic flexural shells.

\section*{Acknowledgements}
{\footnotesize \noindent This research was partially supported by Ministerio de Econom\'ia y Competitividad of Spain, under grants MTM2012-36452-C02-01 and MTM2016-78718-P, with the participation of FEDER.}


\section*{References}
\bibliographystyle{abbrv}

\end{document}